\documentclass[12pt]{amsart}
\usepackage{bm,cite,exscale,enumerate}
\usepackage{graphicx,amssymb,amsfonts}
\usepackage{caption,subcaption}
\usepackage{color,xspace,colortbl}
\usepackage{overpic}

\newcommand{\xx}{\mathbf{x}}
\newcommand{\yy}{\mathbf{y}}
\newcommand{\TT}{\mathbb{T}}

\newcommand{\Z}{\mathbb{Z}}

\newcommand{\R}{\mathbb{R}}

\newcommand{\A}{\mathcal{A}}
\newcommand{\ZZ}{\mathcal{Z}}
\renewcommand{\a}{\mathbf{a}}
\newcommand{\brho}{\bm{\rho}}
\newcommand{\s}{\mathbf{s}}
\renewcommand{\t}{\mathbf{t}}
\newcommand{\gr}{\text{\textnormal{gr}}}
\renewcommand{\H}{\mathcal{H}}
\newcommand{\Vect}{\text{\textnormal{Vect}}}
\renewcommand{\epsilon}{\varepsilon}

\newcommand{\bdry}{\partial}

\newcommand{\be}{\begin{enumerate}}
\newcommand{\ee}{\end{enumerate}}

\newtheorem{thm}{Theorem}[section]

\newtheorem{prop}[thm]{Proposition}
\newtheorem{definition}[thm]{Definition}

\theoremstyle{remark}
\newtheorem{rmk}[thm]{Remark}

\numberwithin{equation}{subsection}

\begin{document}

\title{A topological grading on bordered Heegaard Floer homology}

\author{Yang Huang}
\address{Max-Planck-Institut f\"ur Mathematik, Bonn\\Germany}
\email{yhuang@mpim-bonn.mpg.de}

\author{Vinicius G. B. Ramos}
\address{Universit\'e de Nantes, Nantes \\France}
\email{vinicius.ramos@univ-nantes.fr}

\begin{abstract}
In this paper, we construct a canonical grading on bordered Heegaard Floer homology by homotopy classes of nonvanishing vector fields. This grading is a generalization of our construction of an absolute grading on Heegaard Floer homology and it extends the well-known grading with values in a noncommutative group defined in~\cite{lot}.
\end{abstract}

\maketitle

\section{Introduction}
For a closed oriented 3-manifold $Y$, Ozsv\'{a}th and Szab\'{o}~\cite{osz} defined the Heegaard Floer homology groups $\widehat{HF}(Y)$, $HF^{\infty}(Y)$, $HF^-(Y)$ and $HF^+(Y)$, which are invariants of $Y$. These groups split into direct sums of groups by Spin$^c$ structures, each of which has a relative grading taking values in an appropriate cyclic group. In~\cite{gh}, we constructed a canonical absolute grading for these groups taking values in the set of homotopy classes of oriented 2-plane fields, or equivalently, the set of homotopy classes of nonvanishing vector fields. In this paper, we will define a similar geometric grading on bordered Heegaard Floer homology.

We start by briefly reviewing the construction of bordered Heegaard Floer homology, following~\cite{lot}. Consider a compact oriented 3-manifold $Y$ with non-empty connected boundary. A parametrization of $\partial Y$ is an orientation preserving diffeomorphism $\phi:\partial Y\to F$, where $F$ is a closed oriented surface with a prescribed handle decomposition. According to~\cite{lot}, one can associate to $F$ a differential graded algebra  $\mathcal{A}(F)$. See~\S\ref{sec:alg} for the precise definition of $\mathcal{A}(F)$. Then one defines the so-called {\em type A} and {\em type D} modules of $Y$, denoted by $\widehat{\mathit{CFA}}(Y)$ and $\widehat{\mathit{CFD}}(Y)$. The type $A$ module $\widehat{\mathit{CFA}}(Y)$ is a right $\mathcal{A}_{\infty}$-module over $\A(F)$. That means that there exist maps $$m_l:\widehat{\mathit{CFA}}(Y)\otimes\A(F)^{\otimes (l-1)}\to\widehat{\mathit{CFA}}(Y),$$
satisfying the $\mathcal{A}_{\infty}$-relations, see e.g.~\cite[Eq. (2.6)]{lot}. Here the tensor product is taken over an appropriate ring, as we will review in \S\ref{gr:cfa}. The type $D$ module $\widehat{\mathit{CFD}}(Y)$ is a left differential module over $\A(-F)$, that is, there exists a map $\partial:\widehat{\mathit{CFD}}(Y)\to\widehat{\mathit{CFD}}(Y)$, which squares to $0$ and which satisfies the Leibniz rule with respect to the left action of $\A(-F)$.
It is also shown in~\cite{lot} that if $Y_1$ and $Y_2$ are compact 3-manifolds such that $\partial Y_1=-\partial Y_2$, then there is a homotopy equivalence
\begin{equation}
\Phi:\widehat{\mathit{CFA}}(Y_1)~\widetilde\otimes~\widehat{\mathit{CFD}}(Y_2)\to\widehat{\mathit{CF}}(Y_1\cup_F Y_2).\label{eq:equiv}
\end{equation}
Here $\widetilde\otimes$ denotes the derived tensor product.
For a closed oriented 3-manifold $Y$, we denote by $\textrm{Vect}(Y)$ the set of homotopy classes of nonvanishing vector fields on $Y$.
The goal of this paper is to prove the following theorems.

\begin{thm}\label{thm:alg}
 Given a parameterized surface $F$ as above, there exist a groupoid $G(F)$, with a $\Z$-action denoted by $\lambda^n$ for a given $n\in\Z$, and a grading function $\gr$ with values in $G(F)$ satisfying the following conditions:
 \begin{enumerate}
  \item If $a,b$ are two composable generators of $\A(F)$, then $\gr(a\cdot b)=\gr(a)\cdot \gr(b)$.
  \item If $a$ is a generator of $\A(F)$, then $\gr(\partial a)=\lambda^{-1}\gr(a)$.
 \end{enumerate}
 \end{thm}

 \begin{rmk}
 It turns out that $G(F)$ is by construction a set of co-oriented plane fields on $F\times[0,1]$ modulo homotopy.\footnote{By choosing a Riemannian metric on a 3-manifold, we can identify the set of nonvanishing vector fields with the set of co-oriented plane fields, modulo homotopy, by taking the orthogonal complement.}  The multiplication rule is by the obvious stacking of plane fields when the boundary condition matches. Sometimes these plane fields can be realized as tight contact structures on $F\times[0,1]$ with convex boundary, and $G(F)$ can be mapped into the (universal) contact category $\mathcal{C}(F)$ due to Honda~\cite{ho}.  But we shall not explore this issue any further in this paper.
 \end{rmk}

 \begin{thm}\label{thm:mod}
  For any compact 3-manifold $Y$ with boundary $F$, there exist a set $S(Y)$, admitting a right action by $G(F)$ and a left action by $G(-F)$, and a grading $\gr$ on $\widehat{\mathit{CFA}}(Y)$ and $\widehat{\mathit{CFD}}(Y)$ with values in $S(Y)$ such that
\begin{enumerate}
 \item[(a)] If $x$ is a generator of $\widehat{\mathit{CFA}}(Y)$ and $a_1,\dots,a_l$ are generators of $\A(F)$ such that\\ $m_{l+1}(x;a_1,\dots,a_l)\neq 0$, then $$\gr(m_{l+1}(x;a_1,\dots,a_l))=\lambda^{l-1}\gr(x)\cdot\gr(a_1)\dots\gr(a_l).$$
 \item[(b)] If $x$ is a generator of $\widehat{\mathit{CFD}}(Y)$, then $\gr(\partial x)=\lambda^{-1}\gr(x)$.
\end{enumerate}
\end{thm}

\begin{thm}\label{thm:pair}
Let $Y_1$ and $Y_2$ be compact 3-manifolds such that $\partial Y_1=-\partial Y_2$. Then there exist a set $S(Y_1)\otimes S(Y_2)$ and a map $\Psi:S(Y_1)\otimes S(Y_2)\to \Vect(Y)$ such that
$$\widetilde{\text{\em\gr}}(\Phi(a\otimes b))=\Psi(\text{\em\gr}(a)\otimes\text{\em\gr}(b))$$
for any generators $a$ in $\widehat{\mathit{CFA}}(Y_1)$ and $b$ in $\widehat{\mathit{CFD}}(Y_2)$. Here $\widetilde{\text{\em\gr}}$ denotes the absolute grading in Heegaard Floer homology from~\cite{gh}.
\end{thm}

\begin{rmk}
 Using essentially the same constructions that we will work out in this paper, Theorem~\ref{thm:alg} can be generalized to any surface $F$, not necessarily with connected boundary, using the generalized strands algebra defined by Zarev~\cite{zar}. Both Theorems~\ref{thm:mod} and~\ref{thm:pair} can be generalized to the bimodules $\widehat{\mathit{CFDD}}$, $\widehat{\mathit{CFDA}}$, $\widehat{\mathit{CFAA}}$ constructed in~\cite{lot}, as well as the setting of bordered sutured Floer homology~\cite{zar}, in which case $F\subset \partial Y$, where the inclusion can be strict. The main difference in the construction in the latter case is that one needs to fix a nonvanishing vector field in $\partial Y\setminus F$, similarly to how Spin$^c$ structures are assigned to generators in~\cite{zar}.
\end{rmk}

The paper is organized as follows: In \S\ref{sec:alg}, we first review the definition of the strand algebra $\A(F)$ associated to a parameterized closed surface $F$ following~\cite{lot}. Then we construct the groupoid $G(F)$ in which the grading on $\A(F)$ takes values, and give the proof for Theorem~\ref{thm:alg}. We finish this section by comparing our geometric grading on $\A(F)$ with the previously constructed grading in~\cite{lot}. In \S\ref{sec:mod}, we construct the ``left-$G(-F)$ and right-$G(F)$ bimodule'' $S(Y)$ in which the grading on $\widehat{\mathit{CFA}}(Y)$ and $\widehat{\mathit{CFD}}(Y)$ takes values. Some variations of the standard Pontryagin-Thom construction are made in this section which enable us to compute the relative gradings needed for the proof of Theorem~\ref{thm:mod}. The proof of Theorem~\ref{thm:pair} is provided in \S\ref{sec:pair}.

\section{The grading on the algebra}\label{sec:alg}

In this section, we construct the grading on the algebra $\A(\ZZ)$. This grading takes values in a certain groupoid $G(\ZZ)$. Before defining $G(\ZZ)$ and the grading, we will quickly review the construction of $\A(\ZZ)$. For a more thorough exposition, see~\cite{lot}.

\subsection{The construction of the algebra $\A(\ZZ)$}

The {\em strand algebra} $\A(\ZZ)$ is defined as a subalgebra of $\A(4k)$. As a $\Z/2$-vector space, $\A(4k)$ is generated by partial permutations $(S,T,\phi)$, where $S$ and $T$ are subsets of $\{1,\dots,4k\}$ containing the same number of elements and $\phi:S\to T$ is a bijection such that $\phi(i)\ge i$ for every $i\in S$. We can represent $(S,T,\phi)$ by a diagram with $4k$ points on the left and on the right and with strands connecting the set $S$ on the left with the set $T$ on the right. This diagram is required to have the smallest possible number of crossings. Each crossing corresponds to an inversion, i.e. a pair of points ${i,j}\in\{1,\dots,4k\}$ with $i<j$ and $\phi(i)>\phi(j)$. It follows from this definition that the strands either go up or stay horizontal if we read from left to right. The product of $(S,T,\phi)$ with $(S',T',\phi')$ is defined to be $(S,T',\phi'\circ\phi)$ provided that $T=S'$ and that the number of inversions of $\phi'\circ\phi$ equals the sum of the number of inversions of $\phi$ and $\phi'$. Otherwise, the product is set to be 0. For each subset $S$, one can define an idempotent element $I(S)=(S,S,\mathbb{I}_S)$. One can also define a differential on $\A(4k)$ as follows. For a generator $a$ of $\A(4k)$, let $\partial a$ be the sum over all ways to smooth one crossing of $a$, where we require all the terms of this sum to have exactly one less intersection than $a$. In other words, if smoothing one crossing decreases the number of inversions by more than 1, we set that term to zero.

We denote by $[2k]$ the set $\{1,\dots,2k\}$. A {\em pointed matched circle} $\ZZ$ is a quadruple $(Z,\a,M,z)$ consisting of an oriented circle $Z$, a set of $4k$ points $\a$ in $Z$, a two-to-one function $M:\a\to [2k]$ and a basepoint $z\in Z\setminus\a$. We also require that 0-surgery on $Z$ along the pairs of points that are matched by $M$ yields a single circle. A pointed matched circle gives rise to a surface $F(\ZZ)$ of genus $k$, which we often denote by $F$. The surface $F$ is obtained by starting with a disk whose oriented boundary is $Z$, attaching 1-handles along all the pairs matched by $M$ and attaching a 2-handle to the boundary circle. We observe that we can find a self-indexing Morse function $f:F\to[0,2]$ such that $Z=f^{-1}(3/2)$ and $\a$ is the intersection between $Z$ and the unstable manifolds of the index one critical points. We can identify $[2k]$ with the set of index one critical points $\{p_1,\dots,p_{2k}\}$. We also denote $Z\setminus\{z\}$ by $Z\setminus z$, for simplicity.

By a {\em Reeb chord} $\rho$ we mean an oriented arc on $Z\setminus z$, with the same orientation as $Z$, whose boundary lies in $\a$. We denote by $\rho^-$ the initial endpoint of $\rho$ and by $\rho^+$ its final endpoint. We write $\rho=[\rho^-,\rho^+]$. A set $\bm{\rho}=\{\rho_1,\dots,\rho_m\}$ of Reeb chords is said to be {\em consistent} if both sets $\brho^-:=\{\rho_1^-,\dots,\rho_m^-\}$ and $\brho^+:=\{\rho_1^+,\dots,\rho_m^+\}$ have exactly $m$ elements. A consistent set of Reeb chords $\brho$ gives rise to an element $a_0(\brho)$ in $\A(4k)$ given by
\[a_0(\brho)=\sum_{\substack{S\subset\{1,\dots,4k\}\\S\cap(\brho^-\cup\brho^+)=\emptyset}} (S\cup\brho^-,S\cup\brho^+,\phi_S)\]
where $\phi_S\vert_S=\mathbb{I}$ and $\phi_S(\rho_i^-)=\rho_i^+$ for every $i$. Now, for every $\s\subset[2k]$, we can define the following idempotent
$$I(\s):=\sum_{\substack{S\subset\{1,\dots,4k\}\\M\text{ maps }S\text{ bijectively to }\s}}I(S).$$ We let $\mathcal{I}(\ZZ)$ be the ring of idempotents, which is defined to be the algebra generated by the elements $I(\s)$ for $\s\subset[2k]$. The unit of this algebra is $$\mathbf{I}:=\sum_{\s\subset[2k]} I(\s).$$
We now define the algebra $\A(\ZZ)$ to be the subalgebra of $\A(4k)$ generated by $\mathcal{I}(\ZZ)$ and by the elements $a(\brho):=\mathbf{I}a_0(\brho)\mathbf{I}$, for every consistent set of Reeb chords $\brho$. The algebra $\A(\ZZ)$ is generated as a $\Z/2$-vector space by elements of the form $I(\s)a(\brho)$.
We note that if $I(\s)a(\brho)\neq 0$, then $M|_{\brho^-}$ and $M|_{\brho^+}$ are injective,  $M(\brho^-)\subset\s$ and $(\s\setminus M(\brho^-))\cap M(\brho^+)=\emptyset$. We also observe that the choice of a basepoint $z$ and an orientation on $Z$ induce an ordering on $\mathbf{a}$: if we start from $z$ and follow the positive orientation on $Z$, then $a_i <a_j$ if and only if we meet $a_i$ before $a_j$, where $a_i,a_j \in \mathbf{a}$.

Recall the three different ways that two Reeb chords can intersect. A pair of Reeb chords $\{\rho_1,\rho_2\}$ is said to be {\em interleaved} if $\rho_i^-<\rho_j^-<\rho_i^+<\rho_j^+$ for $\{i,j\}=\{1,2\}$, and {\em nested} if $\rho_i^-<\rho_j^-<\rho_j^+<\rho_i^+$ for $\{i,j\}=\{1,2\}$. The Reeb chords $\rho_1$ and $\rho_2$ are said to {\em abut} if $\rho_1^+=\rho_2^-$. In this case, one defines their {\em join} to be $\rho_1\uplus\rho_2:=[\rho_1^-,\rho_2^+]$. Note that the order of the Reeb chords is important; we will say that $(\rho_1,\rho_2)$ is an abutting pair when $\rho_1^+=\rho_2^-$.

For two sets of Reeb chords $\brho$ and $\bm\sigma$, their {\em join} $\brho\uplus\bm\sigma$ is obtained from the union $\brho\cup\bm\sigma$ where every abutting pair $(\rho,\sigma)$ with $\rho\in\brho$ and $\sigma\in\bm\sigma$ is substituted by $\rho\uplus\sigma$. We recall that if $a(\brho)a(\bm\sigma)\neq 0$, then 
\begin{equation}\label{eq:uplus}
a(\brho\uplus\bm\sigma)=a(\brho)a(\bm\sigma)
\end{equation}

\subsection{The groupoid $G(\ZZ)$}\label{sub:group}
Let $F=F(\ZZ)$. We consider the bundle $TF\oplus\underline{\R}\to F$, where $\underline{\R}$ is the trivial real line bundle. We interpret this bundle as the pullback of the tangent bundle of a three-manifold in which $F$ is embedded, so we call sections of this bundle vector fields on $F$. We will now construct a vector field $v_0':F\to TF\oplus \underline{\R}$. Let $f$ be a self-indexing Morse function compatible with $\ZZ$ as above. Consider its gradient vector field $\nabla f$ and modify it to first eliminate the index zero and index two critical points as follows. Let $\gamma$ be the flow line passing through the basepoint $z$, which connects the index zero critical point to the index two critical point. Let $N(\gamma)$ denote a neighborhood of $\gamma$. Figure~\ref{03}(a) illustrates $\nabla f$ restricted to $N(\gamma)$. We now define a nonvanishing vector field on $N(\gamma)$, which coincides with $\nabla f$ on $\partial N(\gamma)$, as shown in Figure~\ref{03}(b). This picture determines the desired vector field up to homotopy relative to the boundary. This is similar to the construction in \cite[\S 2]{gh}. Let $v_0'$ denote the vector field given by this this construction in $N(\gamma)$ and by $\nabla f$ in the complement of $N(\gamma)$.

\begin{figure}[h]
    \begin{overpic}[scale=.3]{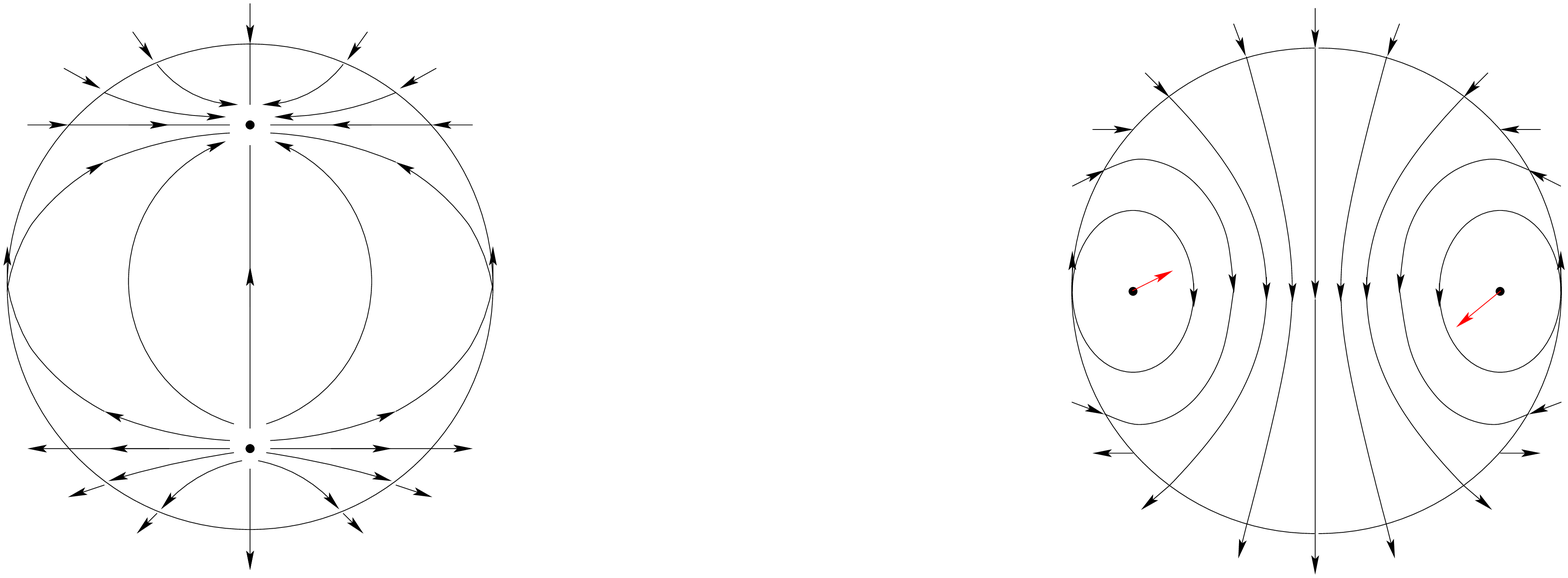}
        \put(14,-4){(a)}
        \put(82,-4){(b)}
    \end{overpic}
    \vspace{4mm}
    \caption{(a) The gradient vector field $\nabla f$ in a neighborhood of the flow line passing through $z$. (b) The nonvanishing vector field in the same neighborhood after modification. The red arrow on the left is pointing into the page and the arrow on the right is pointing out.}
    \label{03}
\end{figure}

Note that each subset $\s\subset[2k]$ corresponds to a set of index one critical points of $f$, under the identification $[2k]=\{p_1,\dots,p_{2k}\}$. We denote by $\bar{\s}$ the subset $[2k]\setminus\s$. For $\s\subset[2k]$, let $\phi_{\s}$ be a bump function which equals 1 at each point of $\s$ and 0 outside of small neighborhoods of each point of $\s$. We denote by $|\s|$ the cardinality of the set $\s$.

\begin{definition}\label{def:v}
 For each $\s\in[2k]$, we define $v_{\s}:F\to TF\oplus \underline{\R}$ to be the vector field given by $$v_{\s}=v_0'+\phi_{\s}\frac{\partial}{\partial t}-\phi_{\bar{\s}}\frac{\partial}{\partial t}.$$ Here $t$ denotes the $\R$-coordinate.
\end{definition}

We can now define the grading set $G(\ZZ)$.

\begin{definition}
For $\s,\t\in[2k]$, such that $|\s|=|\t|$, we define $G(\s,\t)$ to be set of the homotopy classes of nonvanishing vector fields on $F\times[0,1]$ that restrict to $v_{\s}$ on $F\times\{0\}$ and to $v_{\t}$ in $F\times\{1\}$.
We define $G(\ZZ)$ to be the disjoint union of $G(\s,\t)$ for all $\s,\t\subset[2k]$ such that $|\s|=|\t|$.
\end{definition}

Given vector fields $v,w$ on $F\times[0,1]$ such that $v|_{F\times\{1\}}=w|_{F\times\{0\}}$, we can take their concatenation $v\cdot w$, which we see as a vector field on $F\times[0,1]$. So given $[v]\in G(\s,\t)$ and $[w]\in G(\t,\mathbf{u})$, we define their composition by $[v]\cdot [w]:=[v\cdot w]\in G(\s,\mathbf{u})$. We now recall the definition of a groupoid.

\begin{definition}
A groupoid is a category in which every morphism is invertible.
\end{definition}

We observe that $G(\ZZ)$ is a groupoid, whose underlying objects are the vector fields $v_{\s}$ for $\s\subset [2k]$.
The groupoid $G(\ZZ)$ admits a $\Z$-action, defined as follows. We will denote the action of an integer $n\in\Z$ by $\lambda^n$. First observe that, since $\pi_3(S^2)\simeq \Z$, there is a $\Z$-action on the set of homotopy classes of nonvanishing vector fields on a ball $B^3$ relative to its boundary. Our sign convention is such that the Hopf map $S^3\to S^2$ acts on $B^3$ by $\lambda^{-1}$. Note that our sign convention is the opposite of the usual one, but agrees with the one in~\cite{gh}. Let $[v]\in G(\ZZ)$ and fix a ball $B$ in the interior of $F\times[0,1]$. For $n\in\Z$, we define $\lambda^n\cdot [v]$ to be the relative homotopy class of the vector field obtained by the acting on $v|_B$ by $\lambda^n$ and keeping $v$ unchanged outside $B$. We observe that
$$\lambda^n\cdot ([v]\cdot [w])=(\lambda^n\cdot [v])\cdot [w]=[v]\cdot(\lambda^n \cdot [w]).$$

We now recall the Pontryagin-Thom construction in this context. If we fix a trivialization of $T(F\times[0,1])$, we can see every nonvanishing vector field in $F\times[0,1]$ as a map $F\times[0,1]\to S^2$. We observe that two nonvanishing vector fields on $F\times[0,1]$ are homotopic relative to the boundary if, and only if, the corresponding maps $F\times[0,1]\to S^2$ are homotopic relative to the boundary. Now take two maps $v,w:F\times[0,1]\to S^2$ that coincide on $\partial (F\times[0,1])$. We choose a regular value $p$ of both maps and we consider the links $L_v:=v^{-1}(p)$ and $L_w:=w^{-1}(p)$. These links have a framing induced by $v$ and $w$, respectively. Note that $L_v\cap (F\times\{0,1\})$ is a framed 0-manifold in $F\times\{0,1\}$. So $(L_v\cap (F\times\{0,1\}))\times[0,1]$ is a framed one-manifold in $F\times\{0,1\}\times[0,1]$. The links $L_v$ and $L_w$ are said to be relatively framed cobordant if there exists a framed surface $S\subset F\times[0,1]\times[0,1]$, such that
\begin{itemize}
\item[(i)] $\partial S \cap (F\times [0,1]\times\{1\})=L_w\times\{1\}$ as framed submanifolds of $F\times[0,1]\times\{1\}$.
\item[(ii)]  $\partial S \cap (F\times [0,1]\times\{0\})=L_v\times\{0\}$ as framed submanifolds of $F\times[0,1]\times\{0\}$.
 \item[(ii)] $\partial S\cap (F\times\{0,1\}\times[0,1])=(L_v\cap (F\times\{0,1\}))\times[0,1]$ as framed submanifolds of $F\times\{0,1\}\times[0,1]$.
\end{itemize}
The relative version of the Pontryagin-Thom construction says, in this case, that the maps $v$ and $w$ are homotopic relative to the boundary if, and only, $L_v$ and $L_w$ are relatively framed cobordant.

\subsection{A $G(\ZZ)$-grading on $\A(\ZZ)$}\label{construction:G}

Recall that the strand algebra $\A(\ZZ)$ is generated as a $\Z/2$ vector field by all the elements of the form $I(\s)a(\brho)$, where $\s\subset[2k]$ and $\brho=\{\rho_1,\dots,\rho_m\}$ is a consistent set of Reeb chords. For every element $I(\s)a(\brho)\neq 0$, we will define its grading $\gr(I(\s)a(\brho))\in G(\s,\t)$, where $\t=M(\brho^+)\cup(\s\setminus M(\brho^-))$.

For a general $\s\subset[2k]$, in order to draw a picture of $v_{\s}:F\to TF\oplus \underline{\R}$ away from the index 0 and 2 critical points, we will project it to a vector field on $TF$ and decorate the zeros of this vector field using the following convention: an index one critical point $p$ is decorated with $``+"$ if $v_{\s} = \frac{\partial}{\partial t}$ at $p$, and with $``-"$ if $v_{\s} = -\frac{\partial}{\partial t}$ at $p$.

We will define the grading function $\gr$ by steps as follows.

\textsc{Step 1}. Assume that $\brho$ consists of a single Reeb orbit $\rho$, such that $M(\rho^-)\neq M(\rho^+)$. We now construct $\gr(I(\s)a(\rho))$.

We will define a vector field $v_{(\s,\rho)}$ on $F\times[0,1]$ such that $[v_{(\s,\rho)}]\in G(\s,\t)$. Recall that we are identifying a point in $[2k]$ with its corresponding index one critical point. Let $p_i=M(\rho^-)$ and $p_j=M(\rho^+)$. So $\t=\{p_j\}\cup(\s\setminus\{p_i\})$. It follows from our construction in \S\ref{sub:group} that $v_{\s}$ and $v_{\t}$ only differ in small neighborhoods of $p_i$ and $p_j$.

Let $\hat\rho$ be the arc from $p_i$ to $p_j$ consisting of three pieces: the gradient trajectory from $p_i$ to $\rho^-$, the Reeb chord $\rho$ and the gradient trajectory from $p_j$ to $\rho^+$, as shown in Figure~\ref{chord}.

\begin{figure}[h]
    \begin{overpic}[scale=.3]{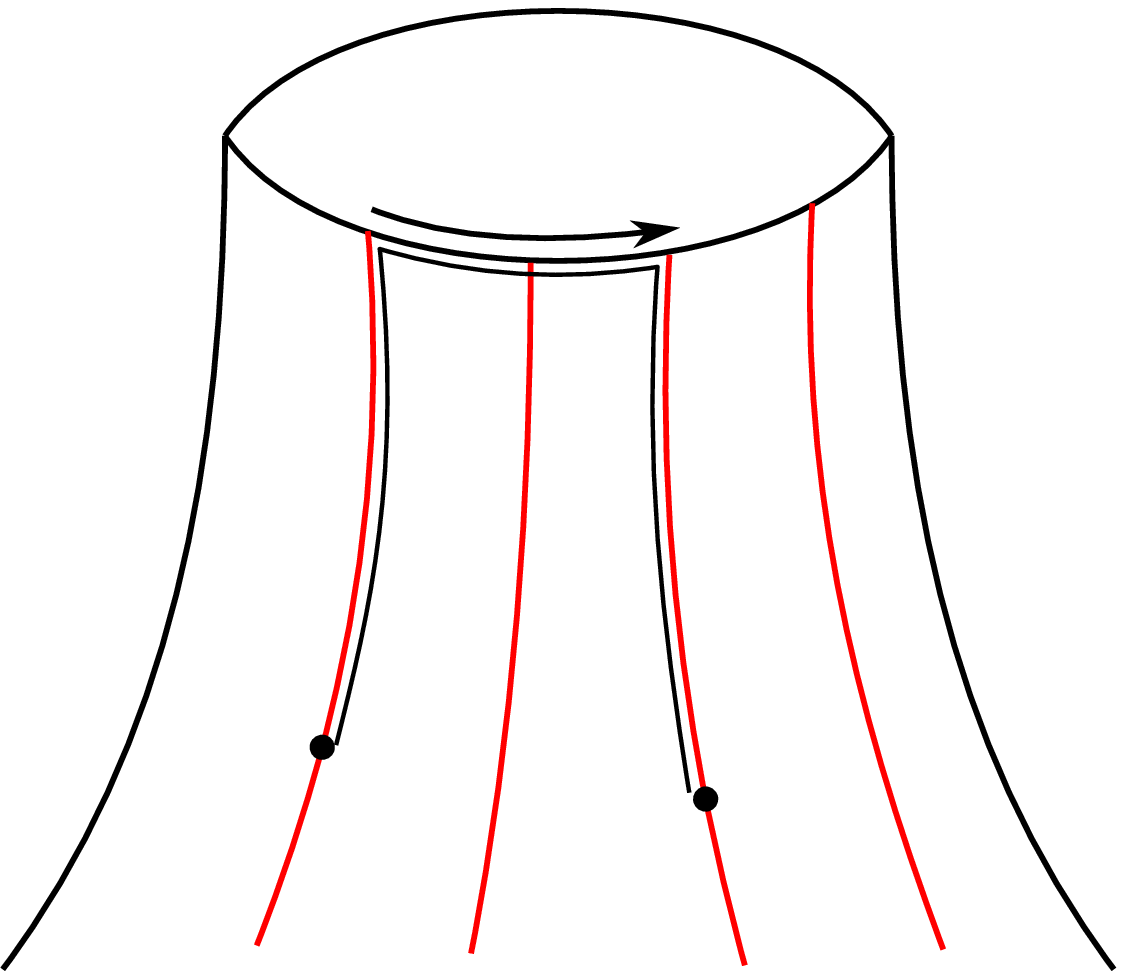}
    \put(18.5,20){\small{$p_i$}}
    \put(65,19){\small{$p_j$}}
    \put(45,70){\small{$\rho$}}
    \end{overpic}
    \caption{Reeb chord $\rho$.}
    \label{chord}
\end{figure}

Let $N(\hat\rho)\subset F$ be a tubular neighborhood of $\hat\rho$. The vector field $v_{\s}$ restricted to $N(\hat\rho)$ is depicted in Figure~\ref{gradingbottom}.

\begin{figure}[h]
    \begin{overpic}[scale=.3]{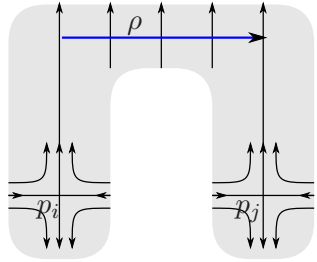}
    \put(9.5,15){\small{$p_i$}}
    \put(74.3,15){\small{$p_j$}}
    \put(39,75){\small{$\rho$}}
    \end{overpic}
    \caption{The neighborhood $N(\hat\rho)$ of $\hat\rho$.}
    \label{gradingbottom}
\end{figure}

Define $v_{(\s,\rho)}$ on $F\times\{0\}$ and $F\times\{1\}$ by setting it equal to $v_{\s}$ and $v_{\t}$, respectively. Since $v_{\s}=v_{\t}$ on the complement of $N(\hat\rho)$, we can extend $v_{(\s,\rho)}$ on $(F \setminus N(\hat\rho))\times[0,1]$, by requiring it to be invariant in the $[0,1]$-direction. The embedding $N(\hat\rho)\subset\R^2$, as shown in Figure~\ref{gradingbottom}, gives rise to a trivialization of $TF|_{N(\hat\rho)}$ and, therefore, we obtain a trivialization of $T(F\times[0,1])|_{N(\hat\rho)\times[0,1]}$. We observe that, under the identification given by this trivialization, $v_{\s}|_{N(\hat\rho)}^{-1}(0,0,1)=p_i$ and $v_{\t}|_{N(\hat\rho)}^{-1}(0,0,1)=p_j$. The points $p_i$ and $p_j$ are framed codimension two submanifolds of $N(\hat\rho)$. By the relative version of the Pontryagin-Thom construction, in order to define a nonvanishing vector field on $N(\hat\rho)\times[0,1]$ with the given boundary condition, it is enough to choose a framed 1-manifold, whose intersection with the boundary is $\{p_i\}\times\{0\}\cup \{p_j\}\times\{1\}$ with the given framing. We choose a framed 1-manifold as follows. Let $\gamma:[0,1]\to F$ be a smoothing of $\hat\rho$ such that $\gamma(0)=p_i$ and $\gamma(1)=p_j$. Let $\tilde\gamma:[0,1]\to F\times[0,1]$ be the arc defined by $\tilde\gamma(t)=(\gamma(t),t)$. Since $F\times\{t\}$ is always transverse to $\tilde\gamma$, the embedding $N(\hat\rho)\subset\R^2$ gives a canonical framing on $\tilde\gamma$. Now, using this framed 1-manifold, the Pontryagin-Thom construction allows us to extend $v_{(\s,\rho)}$ to the interior of $N(\hat\rho)\times[0,1]$. We note that $v_{(\s,\rho)}|_{N(\hat\rho)\times[0,1]}$ is well-defined up to homotopy relative to the boundary.
We now define $\gr(I(\s)a(\rho))$ to be the homotopy class of $v_{(\s,\rho)}$, which is an element of $G(\s,\t)$.

It will be useful later to have a more concrete description of $\gr(I(\s)a(\brho))$. To do so, we view a vector field on $F\times[0,1]$ as a smooth one-parameter family of nonvanishing sections $F\to TF\times \underline{\R}$, indexed by $t\in[0,1]$. We will, in fact, define a family of such sections $\{v_{(\s,\rho)}^{\,t}\}_{t\in[0,1]}$. This family can be explicitly defined by a composition of three bifurcations and necessary isotopies, which we now describe.

Consider the following model situation: Let $\Xi^0$ be a singular vector field on the unit disk $D\subset\R^2$ with two saddle points $p,q$ as depicted in Figure~\ref{bifur}(a). Then there exists a 1-parameter family of vector fields $\Xi^t$, for $0 \leq t \leq 1$, such that
\begin{itemize}
    \item each $\Xi^t$ has only two saddle points which are $p$ and $q$, and $\Xi^t$ is $t$-invariant near $\partial D$,
    \item for exactly one $t$, say $t=1/2$, the vector field $\Xi^t$ has a saddle-saddle connection from $q$ to $p$.
\end{itemize}
See Figure~\ref{bifur} for a pictorial illustration of $\Xi^t$. We call the one-parameter family $\{\Xi^t\}_{t\in[0,1]}$, a {\em bifurcation}. Notice that in the situation of Figure~\ref{bifur}, we decided to fix the unstable trajectories of $p$ and the stable trajectories of $q$ throughout the homotopy, however, we could instead fix the stable trajectories of $p$ and the unstable trajectories of $q$ throughout the homotopy to define another similar one-parameter family of vector fields with the same boundary condition, which we also call a bifurcation.

\begin{figure}[h]
    \begin{overpic}[scale=.25]{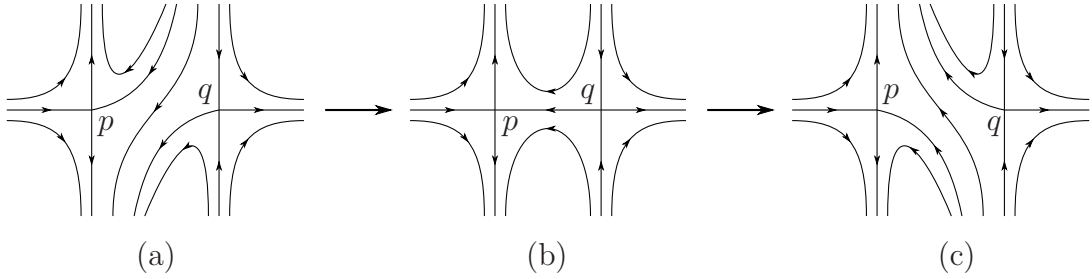}
    \put(8.5,8){$p$}
    \put(45.8,8){$p$}
    \put(81,11){$p$}
    \put(17.7,11){$q$}
    \put(53,11){$q$}
    \put(90.5,8){$q$}
    \put(12,-4.5){(a)}
    \put(48,-4.5){(b)}
    \put(86,-4.5){(c)}
    \end{overpic}
    \vspace{5mm}
    \caption{A bifurcation.}
    \label{bifur}
\end{figure}

We can now define $v_{(\s,\rho)}^{\,t}$ to be equal to $v_{\s}$ for $t\in[0,1]$ in the complement of $N(\hat\rho)$. In $N(\hat\rho)$, we define $v_{(\s,\rho)}^{\,t}$ via a composition of bifurcations and isotopies, as shown in Figure ~\ref{chordgr}. More precisely, there are two saddle points, of different signs, within $N(\hat{\rho})$. We move them in $N(\hat\rho)$ so as to exchange the ``plus'' and the ``minus''. Note that the positive saddle point goes over the negative saddle point as depicted in the third and fourth states of Figure~\ref{chordgr}.

\begin{figure}[h]
    \begin{overpic}[scale=.3]{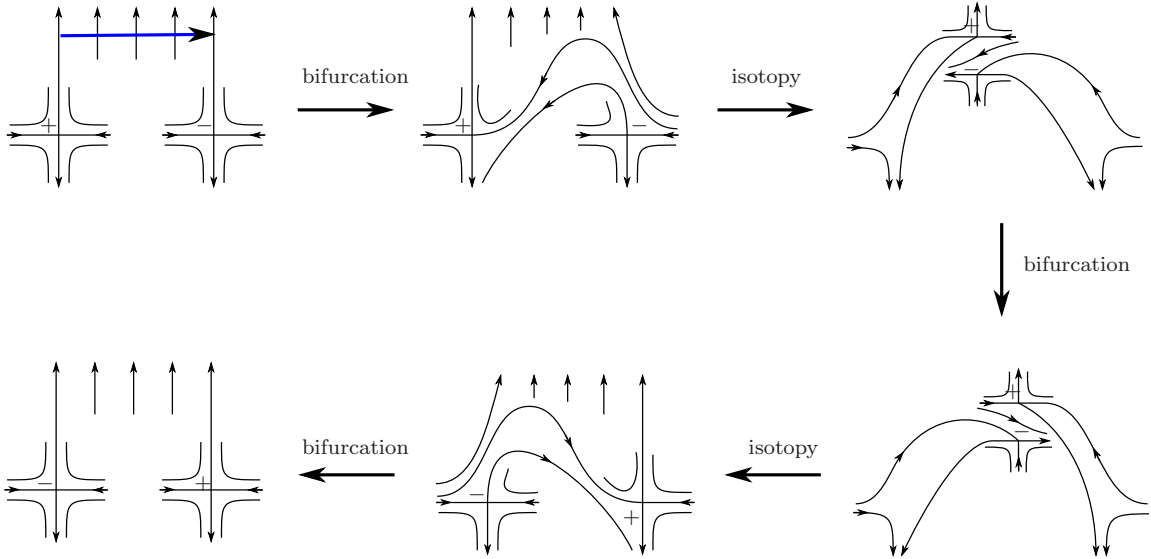}
    \put(26,41.5){\tiny{bifurcation}}
    \put(63.5,41.5){\tiny{isotopy}}
    \put(89,25){\tiny{bifurcation}}
    \put(65,9){\tiny{isotopy}}
    \put(26,9){\tiny{bifurcation}}
    \put(3.1,37.3){\tiny{$+$}}
    \put(16.7,37.3){\tiny{$-$}}
    \put(39.3,37.3){\tiny{$+$}}
    \put(54.7,37.3){\tiny{$-$}}
    \put(83.7,46){\tiny{$+$}}
    \put(83.7,42.2){\tiny{$-$}}
    \put(87.3,14){\tiny{$+$}}
    \put(88.1,10.5){\tiny{$-$}}
    \put(40.5,5){\tiny{$-$}}
    \put(54,3){\tiny{$+$}}
    \put(2.8,6){\tiny{$-$}}
    \put(16.6,6){\tiny{$+$}}
    \end{overpic}
    \caption{A sequence of three bifurcations which defines the grading of $\rho$.}
    \label{chordgr}
\end{figure}

The family $\{v_{(\s,\rho)}^{\,t}\}$ gives rise to a vector field on $F\times[0,1]$, which we denote by $\tilde{v}_{(\s,\rho)}$. As before, the embedding $N(\hat\rho)\subset \R^2$, as in Figure~\ref{gradingbottom}, induces a trivialization of $T(N(\hat\rho)\times[0,1])$. We observe that, in $N(\hat\rho)\times[0,1]$, the framed arc $(\tilde{v}_{(\s,\rho)})^{-1}(0,0,1)$ is isotopic, and hence cobordant, to the arc $\tilde\gamma$. Moreover, their framings coincide under the isotopy. Therefore, by the Pontryagin-Thom construction, $[\tilde{v}_{(\s,\rho)}]=\gr(I(\s)a(\rho))$. Since we had defined $v_{(\s,\rho)}$ up to relative homotopy, we can just take $v_{(\s,\rho)}=\tilde{v}_{(\s,\rho)}$.

\textsc{Step 2.} Now assume that $\brho$ still consists of only one Reeb orbit $\rho$, but $M(\rho^-)= M(\rho^+)$.

Let $p=M(\rho^-)=M(\rho^+)$ and let $\hat\rho'$ be the union of $\rho$ and the flow lines connecting $p$ to $\rho^-\cup\rho^+$. We construct a one-parameter family $\{\Theta^t\}_{t\in[0,1]}$ of vector fields on $F$ as follows. Set $\Theta^0=v_{\s}$ and $\Theta^t\equiv\Theta^0$, for $t\in[0,1]$, outside $N(\hat\rho')$. Fix a small $\epsilon>0$. For $t\in[0,\epsilon]$, define $\Theta^t$ in $N(\hat\rho')$ to be the homotopy which creates an extra pair of singular points near $p$ decorated with negative signs, along the unstable trajectories of $p$ as depicted in Figure~\ref{cancelpair}. More precisely, under the projection to $TF$, we create a pair of canceling critical points $\mu$ of index one and $\nu$ of index two, lying on the flow line of $\nabla f$ connecting $p$ to $\rho^+$. Consider the (broken) arc $\hat\rho$ from $p$ to $\mu$, which is the union of the trajectory from $p$ to $\rho^-$, the Reeb chord $\rho$ and the trajectory from $\mu$ to $\rho^+$. Now we can repeat the method from Step 1 for $\Theta^{\epsilon}|_{N(\hat\rho)}$ and obtain a homotopy $\Theta^{t}|_{N(\hat\rho)}$ for $t\in[\epsilon,1-\epsilon]$, which exchanges the signs of the index one critical points. We define $\Theta^{t}$ in $N(\hat\rho')\setminus N(\hat\rho)$ to be equal to $\Theta^{\epsilon}$, for all $t\in[\epsilon,1-\epsilon][$. Now, for $t\in[1-\epsilon,1]$, let $\Theta^t|_{N(\hat\rho')}$ be the homotopy which cancels the extra pair of ``negative'' singular points. The family $\{\Theta^t\}_{t\in[0,1]}$ gives rise to a vector field, which is again denoted by $v_{(\s,\rho)}$. Finally $\textrm{gr}(I(\s)a(\rho))$ is defined to be the homotopy class of $v_{(\s,\rho)}$.

\begin{figure}[h]
    \begin{overpic}[scale=.35]{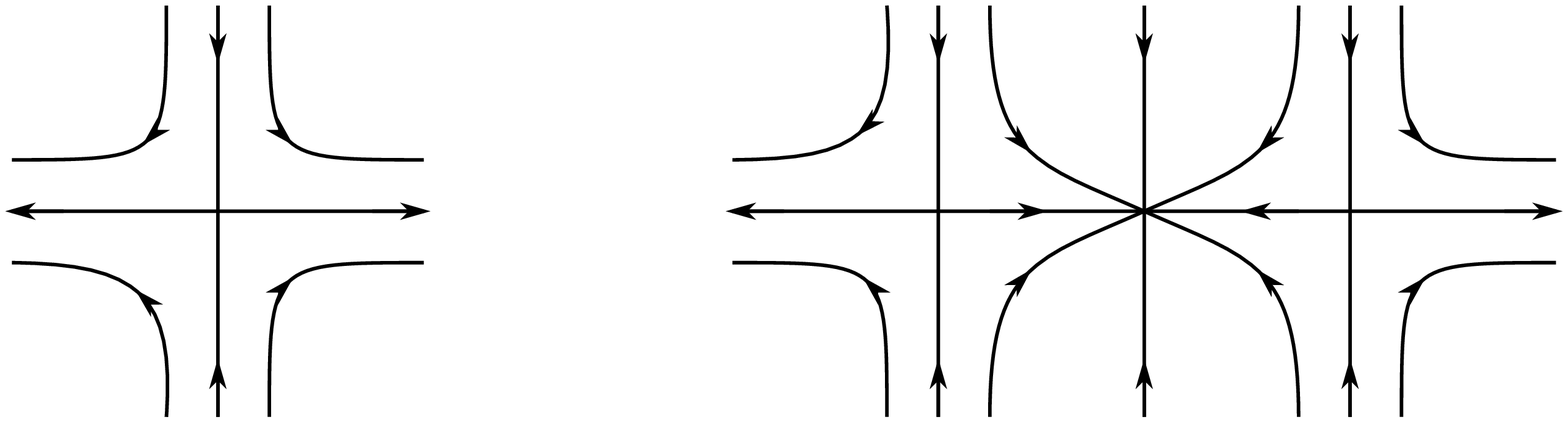}
    \put(14,14.3){\small{$+$}}
    \put(60.5,14.5){\small{$+$}}
    \put(73.7,15.5){\small{$-$}}
    \put(87,14.3){\small{$-$}}
    \put(10,10.5){$p$}
    \put(57,10.5){$p$}
    \put(70,9){$\nu$}
    \put(83.5,10.5){$\mu$}
    \end{overpic}
    \caption{Creating a canceling pair of critical points with negative sign.}
    \label{cancelpair}
\end{figure}

\textsc{Step 3.} The general case.

Suppose $\bm{\rho}=\{\rho_1,\cdots,\rho_l\}$. We define an ordering on $\bm{\rho}$ by setting $\rho_i < \rho_j$ whenever $\rho^{+}_i > \rho^{+}_j$ in $\mathbf{a}$. Up to re-ordering, we may assume that $\rho_1<\rho_2< \cdots <\rho_l$.
We want to define a relative homotopy class $\textrm{gr}(I(\s)a(\bm{\rho})) \in G(v_{\s},v_{\t})$, where $\mathbf{t}=M(\brho^+)\cup(\s\setminus M(\brho^-))$. First, for every point in $(M(\brho^-)\cap M(\brho^+))\setminus M(\brho^-\cap\brho^+)$, we create a pair of canceling ``negative'' singular points, as follows. If $p=M(\rho_i^+)=M(\rho_j^-)$, then we create a pair of ``negative'' singular points on the flow line connecting $p$ to $\rho_i^+$, as in Step 2. This construction gives rise to a vector field $v_{\epsilon}$ in $F\times[0,1]$ similar to $\{\Theta|_t\}_{t\in[0,\epsilon]}$ from Step 2. We also consider the vector field $v_{-\epsilon}$, which corresponds to canceling the ``negative'' singular points added to $v_{\t}$. Now consider the arcs $\hat\rho_i$ associated to $\rho_i$ as before, namely, $\hat\rho_i$ is the union of $\rho_i$ with the gradient trajectories connecting index one critical points to $\rho^{-}_i \cup \rho^{+}_i$. Note that $\hat\rho_i$ always connects a ``positive'' saddle to a ``negative'' saddle. Now we define $v_{(\s,\bm\rho)}^1$ to be the vector field which equals $v_{(\s.\rho_1)}$ on $N(\hat\rho_1)\times[0,1]$ and which is $[0,1]$-invariant elsewhere.
We repeat the same procedure for $\rho_2,\dots,\rho_l$, such that for every $i\ge 2$, the vector field $v_{(\s,\brho)}^i$ corresponding to $\rho_i$ is $[0,1]$-invariant in the complement of $N(\hat\rho_i)\times[0,1]$ and, in $N(\hat\rho_i)\times[0,1]$, it is given by the description as in Step 3. In particular,
$$v_{(\s,\brho)}^{i-1}|_{F\times\{1\}}=v_{(\s,\brho)}^i|_{F\times\{0\}}.$$
Let $v_{(\s,\bm\rho)}$  be the concatenation
\begin{equation}\label{eq:v}
v_{(\s,\brho)}:=v_{\epsilon}\cdot v_{(\s,\brho)}^1 \dots v_{(\s,\brho)}^l\cdot v_{-\epsilon}.
\end{equation}
Finally, we define $\text{gr}(I(\s)a(\brho))$ to be the relative homotopy class of $v_{(\s,\brho)}$, which is an element of $G(\s,\t)$.

\subsection{The properties of the grading on $\A(\ZZ)$}
We now show that the grading we constructed in the previous subsection satisfies the desired properties.

\begin{prop}\label{prop:propgrad}
The grading function $\text{\em gr}: \mathcal{A}(\ZZ) \to G(\ZZ)$ constructed above defines a grading on the differential graded algebra $\mathcal{A}(\ZZ)$, i.e., it satisfies the following:
\begin{itemize}
    \item For any two sets of Reeb chords $\bm\rho,\bm\sigma$, if $I(\s)a(\bm\rho)I(\t) a(\bm\sigma) \neq 0$, then $$\text{{\em gr}}(I(\s)a(\bm\rho)) \cdot\text{{\em gr}}(I(\t)a(\bm\sigma)) = \text{{\em gr}}(I(\s)a(\bm\rho) I(\t)a(\bm\sigma)),$$
    \item For any $\brho$, if $\partial (I(\s)a(\bm\rho))\neq 0$, then $$\text{{\em gr}}(\partial( I(\s)a(\bm\rho)))=\lambda^{-1} \cdot \text{{\em gr}}(I(\s)a(\bm\rho)).$$
\end{itemize}
\end{prop}

\begin{proof}
We shall use the Pontryagin-Thom construction to prove both assertions of the proposition. The proof will be divided in three steps.

\textsc{Step 1}: For a pair $(\s,\brho)$, we define a submanifold $Q_{(\s,\brho)}\subset F\times[0,1]$ and we relate it to $\gr(I(\s)a(\brho))$.

We denote by $N(z)$ and a small neighborhood of $z$ and let $\mathfrak{N}$ be a small neighborhood of $Z\setminus N(z)$. For each index one critical point $p_i\in[2k]$, we denote by $H_i\subset F$ the corresponding 1-handle. 
We fix an orientation-preserving embedding $\mathfrak{N}\hookrightarrow\R^2$ that restricts to an orientation-preserving embedding $Z\setminus N(z)\hookrightarrow \R\times\{0\}$.
Let $\bar{\mathfrak{N}}:=\mathfrak{N}\cup H_i$. We can construct an immersion $\bar{\mathfrak{N}}\looparrowright \R^2$ whose restriction to $\mathfrak{N}$ is the previous embedding and such that the core of $H_i$ maps to a half-circle, see Figure~\ref{pontmfd}(a).
This immersion induces a trivialization of $TF|_{\bar{\mathfrak{N}}}$, which gives rise to a $[0,1]$-invariant trivialization of $T(F\times[0,1])$ over $\bar{\mathfrak{N}}\times[0,1]$. It can be extended to a $[0,1]$-invariant trivialization of $T(F\times[0,1])$, which we denote by $\tau$.

\begin{figure}[ht]
 \begin{overpic}[scale=.9]{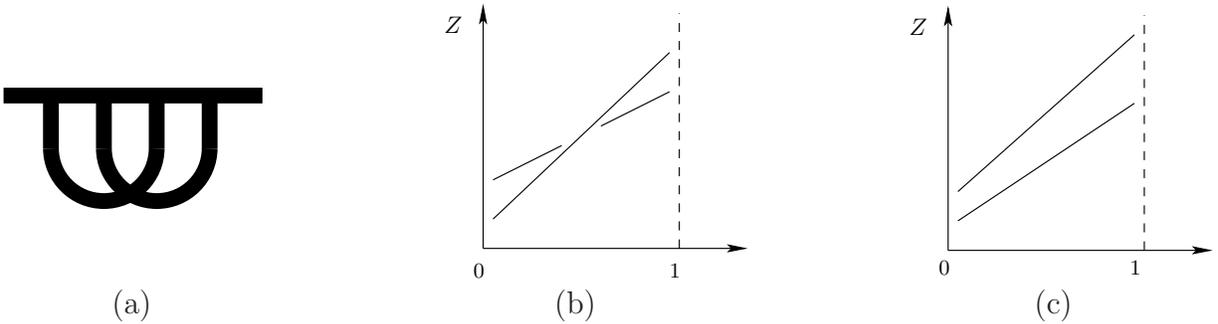}
  \put(10,-2){(a)}
   \put(46,-2){(b)}
   \put(85,-2){(c)}
    \end{overpic}
\vspace{3mm}
    \caption{(a) The image of $\bar{\mathfrak{N}}$ in $\R^2$. (b) A nested pair. (c) An interleaved pair.}
    \label{pontmfd}
\end{figure}

Fix a generator $I(\s)a(\bm\rho) \in \mathcal{A}(\ZZ)$. We now define a one-manifold $Q_{(\s,\brho)}$ in $\bar{\mathfrak{N}}\times [0,1]$.  We write $\brho=\{\rho_1,\dots,\rho_l\}$ and we let $\gamma_i\subset \bar{\mathfrak{N}}$ be the curve obtained by smoothing the union\footnote{If $M(\rho_i^+)\neq M(\rho_j^-)$ for every $j$ or $\rho_i^+=\rho_j^-$ for some $j$, then $\gamma_i$ is just $\hat\rho_i$. Otherwise, to obtain $\gamma_i$, we need to add a small segment connecting $M(\rho_i^+)$ to the ``negative'' index one critical point corresponding to $\rho_i^+$.} of $\rho_i$ with the gradient flow trajectories connecting $M(\rho_i^-)\cup M(\rho_i^+)$ to $\rho_i^-\cup\rho_i^+$. For every nested pair $\{\rho_i,\rho_j\}$, we isotope $\gamma_i$ and $\gamma_j$ slightly in $\mathfrak{N}$ so that they do not intersect in $\mathfrak{N}$.
The orientation of $Z$ induces an orientation of $\gamma_i$. Now we parametrize $\gamma_i$, obtaining an injection $\gamma_i:[0,1]\to \mathfrak{N}$. We define a partial order $<$ on $\mathfrak{N}$ seen as a subset of $R^2$, by saying that $(x_1,y_1)<(x_2,y_2)$ if $x_1<x_2$. For every interleaved pair $\{\rho_i,\rho_j\}$ with $\rho_i^-<\rho_j^-$, we can assume that $\gamma_i(t)<\gamma_j(t)$, whenever both $\gamma_i(t)$ and $\gamma_j(t)$ belong to $\mathfrak{N}$. 

Now we define arcs $\tilde{\gamma}_i$ on $F\times [0,1]$ by $\tilde{\gamma}_i(t)=(\gamma_i(t),t)$. It follows from our construction that the arcs $\tilde{\gamma}_i$ are all pairwise disjoint. We define $Q_{(\s,\brho)}$ to be the union of $\tilde{\gamma}_i$ for all $i$ and the constant arcs $p\times[0,1]$, where $p\in\s\setminus M(\brho^-)$. Up to a small isotopy of $Q_{(\s,\brho)}$, we can assume that the projection of $Q_{(\s,\brho)}\cap(\mathfrak{N}\times[0,1])$ to $Z\times[0,1]$ has minimal number of intersections. In fact, there will be one intersection for each nested pair. Notice that, for a nested pair $\{\rho_i,\rho_j\}$ with $\rho_i^+<\rho_j^+$, the image of $\tilde{\gamma}_i$ goes under the image of $\tilde{\gamma}_j$, where the height is the coordinate in $\mathfrak{N}\subset\R^2$ corresponding to the second factor of $\R$. We keep track of which strand goes over which on the projection of $Q_{(\s,\brho)}\cap(\mathfrak{N}\times[0,1])$, as in Figure~\ref{pontmfd}(b). This projection can be seen as the strand diagram corresponding to $\brho$. See Figure~\ref{pontmfd}(b),(c) for examples where $\brho$ is a pair of Reeb chords. We obtain a diagram for $I(\s)a(\brho)$ if we include two dotted horizontal lines corresponding to the two points in $M^{-1}(p)$, for every $p\in \s\setminus M(\brho^-)$.

If we want to represent a concatenation $Q_{(\s,\brho)}\cdot Q_{(\t,\bm\sigma)}$, we will substitute the pair of dotted lines in either of the two diagrams by a solid line whenever one of the two corresponding points gets moved by the other set of Reeb chords. For an example, see Figure~\ref{dotted}.

\begin{figure}[ht]
 \begin{overpic}[scale=1]{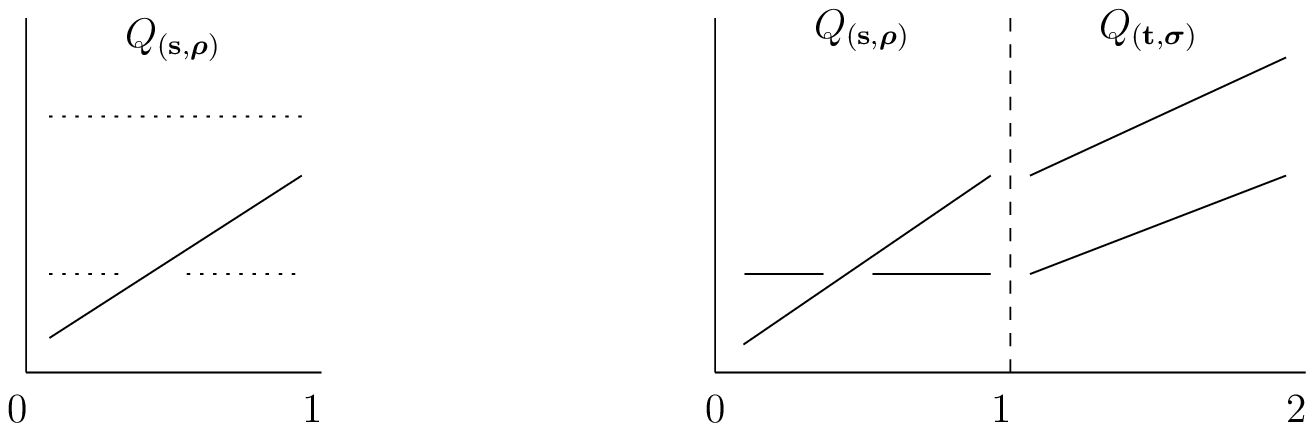}
    \end{overpic}
    \caption{A diagram for $Q_{(\s,\brho)}$ and for a concatenation $Q_{(\s,\brho)}\cdot Q_{(\t,\bm\sigma)}$.}
    \label{dotted}
\end{figure}

Let $v_{(\s,\brho)}$ be the vector field constructed in \S\ref{construction:G} whose homotopy class is $\gr(I(\s)a(\brho))$. Recall that $v_{(\s,\brho)}(x,t)=v_{\s}(x)$ for every $(x,t)\not\in\bar{\mathfrak{N}}\times[0,1]$. As usual, we let $\t=(\s\setminus M(\brho^-))\cup M(\brho^+)$. Using $\tau$, we see $v_{(\s,\brho)}$ as a map $v_{(\s,\brho)}:F\times[0,1]\to S^2$ and $v_{\s}$ and $v_{\t}$ as maps $v_{\s},v_{\t}:F\to S^2$. We can perturb $v_{(\s,\brho)}$, $v_{\s}$ and $v_{\t}$ slightly so that $(0,0,1)$ is a regular value of all these maps. Then \begin{align*}
v_{\s}^{-1}(0,0,1)&=\s\cup P,\\
v_{\t}^{-1}(0,0,1)&=\t\cup P,                                                                                                                                                                                                                                                                                                                                                                                                                                                                                                                                                                                                                                                                                                                                                                                                                                                                                                                                                                                                                                                                                                                                                        \end{align*}
where $P$ is a set of points in $F\setminus\bar{\mathfrak{N}}$. Now we want to compute $v_{(\s,\brho)}^{-1}(0,0,1)$. This is a one-manifold with boundary on $\partial (F\times[0,1])$. Namely $$\partial (v_{(\s,\brho)}^{-1}(0,0,1))=v_{\t}^{-1}(0,0,1)\cup (-v_{\s}^{-1}(0,0,1))=((\t\cup P)\times\{1\})\cup (-(\s\cup P)\times\{0\}).$$
We observe that 
$$v_{(\s,\brho)}^{-1}(0,0,1)\setminus (\bar{\mathfrak{N}}\times[0,1])=P\times[0,1].$$
So it remains to compute $(v_{(\s,\brho)}|_{\bar{\mathfrak{N}}\times[0,1]})^{-1}(0,0,1)$. We recall that $v_{(\s,\brho)}$ is defined as the concatenation of the vector fields $v_{(\s,\brho)}^1,\dots, v_{(\s,\brho)}^l$ and small perturbations at the beginning and at the end. We notice that, up to small isotopies, $(v_{(\s,\brho)}^i|_{\bar{\mathfrak{N}}\times[0,1]})^{-1}(0,0,1)$ is a braid given by the union of $\tilde{\gamma}^i$ with horizontal chords of the form $\{p\}\times[0,1]$, where $p=p_j$ or $p$ is close to $p_j$ for some $j$. The preimage of $(0,0,1)$ under the pertubation vector fields $v_{\epsilon}$ and $v_{-\epsilon}$ are braids with some horizontal chords and some nearly horizontal chords corresponding to the perturbations. After concatenating all those braids and performing an isotopy, we obtain
$$(v_{(\s,\brho)}|_{\bar{\mathfrak{N}}\times[0,1]})^{-1}(0,0,1)=Q_{(\s,\brho)}.$$
We observe that the ordering of the vector fields $v_{(\s,\brho)}^i$ in the concatenation is very important here. In particular, for every interleaved pair $\{\rho_i,\rho_j\}$ with $\rho_i^+>\rho_j^+$, the vector field $v_{(\s,\brho)}^i$ is on the left of $v_{(\s,\brho)}^j$, which ensures that the corresponding braid is isotopic to $\tilde{\gamma}_i\cup\tilde{\gamma}_j$. Figure~\ref{rho_sigma} illustrates this difference. Case (3) depicts what we obtain using the prescribed ordering and case (4) shows what we would obtain if we switched the order of $\rho_i$ and $\rho_j$. Note that the submanifolds obtained in these two cases are not isotopic.

The framing on $Q_{(\s,\brho)}$ induced by $v_{(\s,\brho)}$ is trivial and has a standard form near every $p_i$\footnote{The standard framing is a rotation by $\pi$ in $N(p_i)$ either positively or negatively depending on whether $p_i=M(\rho^-)$ or $p_i=M(\rho^+)$ for the corresponding Reeb chord $\rho$, but it does not depend on anything else.}. The framed manifold $v_{(\s,\brho)}^{-1}(0,0,1)$ is called the {\em Pontryagin} submanifold of $v_{(\s,\brho)}$.

\textsc{Step 2}: We now prove part (a) of the proposition.


Let $\brho=\{\rho_1,\dots,\rho_n\}$ and $\bm\sigma=\{\sigma_1,\dots,\sigma_m\}$, where we order the Reeb chords as in \S\ref{construction:G}. We assume that $I(\s)a(\brho) I(\t) a(\bm\sigma)\neq 0$, for $|\s|=|\t|$. We will show that
\begin{equation*}
\gr(I(\s)a(\brho))\cdot \gr(I(\t) a(\bm\sigma))=\gr(I(\s)a(\brho) I(\t) a(\bm\sigma)).
\end{equation*}
By~\eqref{eq:uplus}, $I(\s)a(\brho) I(\t) a(\bm\sigma)=I(\s)a(\brho\uplus\bm\sigma)$. So it is enough to show that
\begin{equation}
 v_{(\s,\brho)}\cdot v_{(\t,\bm\sigma)}\cong v_{(\s,\brho\uplus\bm\sigma)}.\label{eq:concat}
\end{equation}
Here $\cdot$ denotes the concatenation of vector fields in $F\times[0,1]$ and $\cong$ denotes homotopy relative to the boundary.

By the Pontryagin-Thom construction, in order to prove~\eqref{eq:concat}, it is enough to show that the Pontryagin submanifolds of both sides are framed cobordant relative to the boundary. It follows from Step 1 that the Pontryagin submanifold of $v_{(\s,\brho)}\cdot v_{(\t,\bm\sigma)}$ is obtaining by concatenating $Q_{(\s,\brho)}$ and $Q_{(\t,\bm\sigma)}$ and taking its union with $P\times[0,1]$. Since the Pontryagin submanifold of $v_{(\s,\brho\uplus\bm\sigma)}$ is $Q_{(\s,\brho\uplus\bm\sigma)}\cup P\times[0,1]$, it is enough to see that $Q_{(\s,\brho)}\cdot Q_{(\t,\bm\sigma)}$ is framed cobordant to $Q_{(\s,\brho\uplus\bm\sigma)}$ in $\bar{\mathfrak{N}}\times[0,1]$. We will first prove that $Q_{(\s,\brho)}\cdot Q_{(\t,\bm\sigma)}$ is isotopic to $Q_{(\s,\brho\uplus\bm\sigma)}$ relative to the boundary.

We write $\brho=\{\rho_1,\dots,\rho_n\}$ and $\bm\sigma=\{\sigma_1,\dots,\sigma_m\}$, where we order the Reeb chords as in \S\ref{construction:G}.
For Reeb chords $\rho_i\in\brho$ and $\sigma_j\in\bm\sigma$, let $\tilde{\gamma}_{\rho_i}$ and $\tilde{\gamma}_{\sigma_j}$ denote the corresponding components of $Q_{(\s,\brho)}$ and $Q_{(\t,\bm\sigma)}$, respectively. If $\rho_i$ and $\sigma_j$ abut for some $i,j$, then the concatenation $\tilde{\gamma}_{\rho_i}\cdot\tilde{\gamma}_{\sigma_j}$ is clearly isotopic to the chord corresponding to $\rho_i\uplus\sigma_j$ in $Q_{(\s,\brho\uplus\bm\sigma)}$. Whenever $\rho_i$ does not abut with any $\sigma_j$, we see that $\tilde{\gamma}_{\rho_i}$ concatenated with a horizontal chord in $Q_{(\t,\bm\sigma)}$ is isotopic to the chord corresponding to $\rho_i$ in $Q_{(\s,\brho\uplus\bm\sigma)}$. Similarly, for every $\sigma_j$ such that there is no $\rho_i$ with $\rho_i^+=\sigma_j^-$, the chord obtained from concatenating the appropriate horizontal chord in $Q_{(\s,\brho)}$ with $\tilde{\gamma}_{\sigma_j}$ is isotopic to the chord corresponding to $\sigma_j$ in $Q_{(\s,\brho\uplus\bm\sigma)}$. So every connected component of the concatenation $Q_{(\s,\brho)}\cdot Q_{(\t,\bm\sigma)}$ is isotopic to the corresponding component of $Q_{(\s,\brho\uplus\bm\sigma)}$. It remains to check that we can perform each of these isotopies in the complement of the others, so that we obtain an isotopy from $Q_{(\s,\brho)}\cdot Q_{(\t,\bm\sigma)}$ to $Q_{(\s,\brho\uplus\bm\sigma)}$.

First let us consider a pair of chords $\rho_i$ and $\sigma_j$ that do not abut. If they do not intersect, then the corresponding isotopies can clarly be chosen to have disjoint supports.
If $\rho_i$ and $\sigma_j$ intersect, then we have one of the following four possibilities:
\begin{enumerate}
 \item $\rho_i^-<\sigma^-<\sigma^+<\rho_i^+$,
\item $\sigma^-<\rho_i^-<\rho_i^+<\sigma^+$,
\item $\sigma^-<\rho_i^-\le\sigma^+<\rho_i^+$,
\item $\rho_i^-<\sigma^-<\rho_i^+<\sigma^+$.
\end{enumerate}
In cases (1) and (2), $\{\rho_i,\sigma\}$ is a nested and in cases (3) and (4), $\{\rho_i,\sigma\}$ is interleaved. For an example of each of these cases, see Figure~\ref{rho_sigma}. We note that (3) includes the case when $(\sigma_j,\rho_i)$ is an abutting pair.
\begin{figure}[ht]
 \begin{overpic}[scale=1.3]{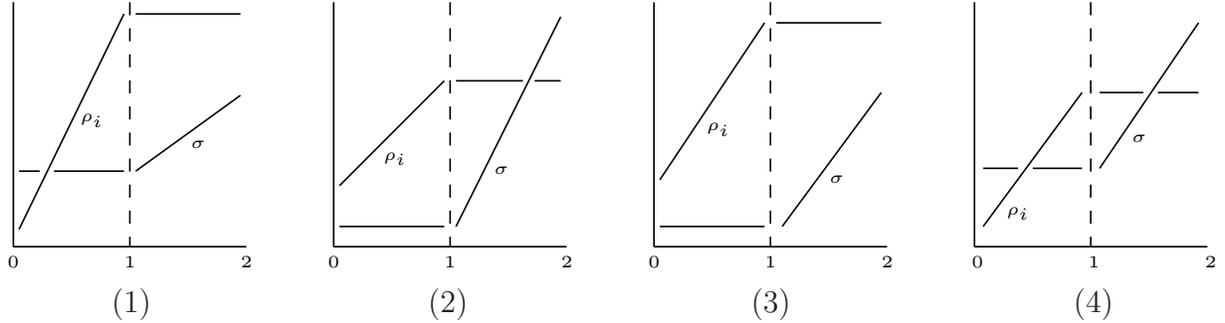}
  \put(9,-3){(1)}
   \put(35,-3){(2)}
   \put(61.5,-3){(3)}
\put(88,-3){(4)}
    \end{overpic}
\vspace{4mm}
    \caption{The four cases when $\rho_i$ and $\sigma$ intersect, but $(\rho_i,\sigma)$ is not abutting.}
    \label{rho_sigma}
\end{figure}
We observe that, in case (4), $a(\brho)a(\bm\sigma)=0$, so by our assumption, it cannot occur. In the three other cases, we see that the corresponding isotopies can be chosen to have disjoint supports.

Now, let $(\rho_i,\sigma_j)$ be an abutting pair. If the diagram for $Q_{(\s,\brho)}\cdot Q_{(\t,\bm\sigma)}$ has no crossing along $\tilde{\gamma}_{\rho_i}\cdot\tilde{\gamma}_{\sigma_j}$, then we can clearly choose an isotopy from $\tilde{\gamma}_{\rho_i}\cdot\tilde{\gamma}_{\sigma_j}$ to the chord corresponding to $\rho_i\uplus\sigma_j$ in $Q_{(\s,\brho\uplus\bm\sigma)}$, whose support is in the complement of the other isotopies. If there exists a chord in $Q_{(\s,\brho)}\cdot Q_{(\t,\bm\sigma)}$ that crosses $\tilde{\gamma}_{\rho_i}\cdot\tilde{\gamma}_{\sigma_j}$ twice, then $a(\brho)a(\bm\sigma)\neq 0$. Hence, we only need to consider the case when there exists a chord in $Q_{(\s,\brho)}\cdot Q_{(\t,\bm\sigma)}$ that crosses $\tilde{\gamma}_{\rho_i}\cdot\tilde{\gamma}_{\sigma_j}$ exactly once. But in this case, the crossing corresponds either to a crossing of $Q_{(\s,\brho)}$ or to a crossing of $Q_{(\t,\bm\sigma)}$. So we can isotope $\tilde{\gamma}_{\rho_i}\cdot\tilde{\gamma}_{\sigma_j}$ and the corresponding chord without changing any crossings. Therefore these isotopies can be chosen to have disjoint supports. We conclude that $Q_{(\s,\brho)}\cdot Q_{(\t,\bm\sigma)}$ is isotopic to $Q_{(\s,\brho\uplus\bm\sigma)}$.

We also observe that these isotopies preserve the trivial framing, since they can be chosen so that the braids are always transverse to $F\times\{t\}$. Therefore $Q_{(\s,\brho)}\cdot Q_{(\t,\bm\sigma)}$ and $Q_{(\s,\brho\uplus\bm\sigma)}$ are framed homotopic relative to the boundary. Hence
$$\gr(I(\s)a(\brho))\cdot \gr(I(\t) a(\bm\sigma))=\gr(I(\s)a(\brho) I(\t) a(\bm\sigma)).$$

\textsc{Step 3}: We prove part (b) of the proposition.

Let $I(\s)a(\brho)$ be a generator of $\mathcal{A}(F)$ and let $v_{(\s,\brho)}$ and $Q_{(\s,\brho)}$ be as in Step 1 above. Recall that the differential of $I(\s)a(\brho)$ is the sum of all ways of resolving one crossing of $I(\s)a(\bm\rho)$. Let $I(\s)a(\tilde{\brho})$ be one of the terms in $\partial(I(\s)a(\brho))$. So $Q_{(\s,\tilde{\brho})}$ is obtained from $Q_{(\s,\brho)}$ by resolving one crossing. Note that both submanifolds have the trivial framing induced by the immersion $\bar{\mathfrak{N}}\looparrowright \R^2$. Figures~\ref{surgery}(a),(b) show the submanifolds and their framings for two cases, when the crossing that is being resolved concerns two Reeb chords and when it concerns a Reeb chord and a horizontal strand. We see that, in both cases, $Q_{(\s,\tilde{\brho})}$ is the result of performing a 0-surgery to $Q_{(\s,\brho)}$, as in Figure~\ref{surgery}(c). So $Q_{(\s,\tilde{\brho})}$ are $Q_{(\s,\brho)}$ cobordant. We notice that using this cobodism, the trivial framing on $Q_{(\s,\brho)}$ induces a framing on $Q_{(\s,\tilde{\brho})}$ that differs from the trivial one by 1, see Figure~\ref{surgery}(d). Recall that, by our sign convention, a clockwise turn changes the framing by $+1$. Therefore 
$$\gr(I(\s)a(\tilde{\brho}))=\lambda^{-1}\gr (I(\s)a(\brho)).$$
\begin{figure}
 \begin{overpic}[scale=0.7]{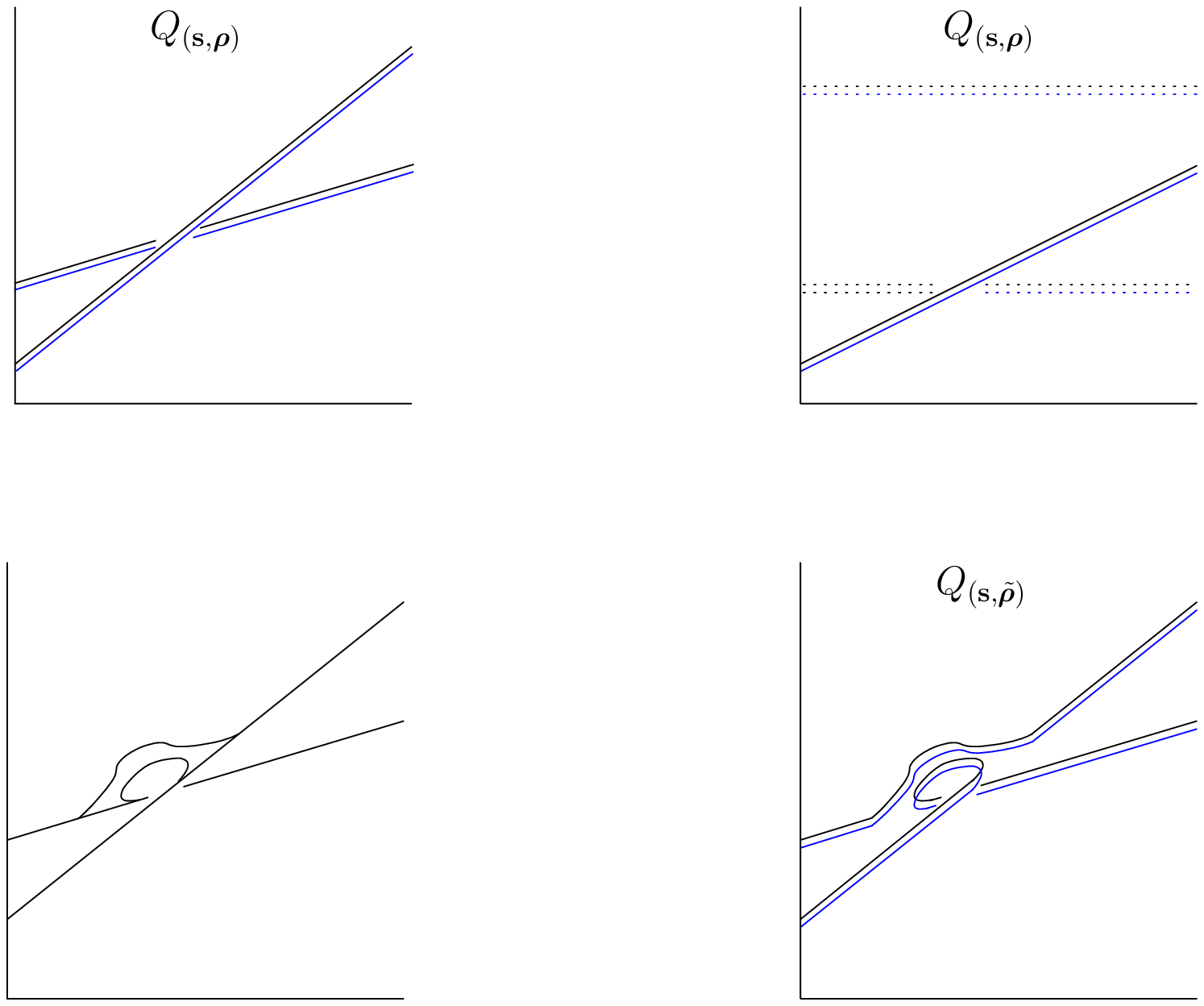}
  \put(14.5,45){(a)}
   \put(83,45){(b)}
   \put(14.5,-4){(c)}
\put(83,-4){(d)}
    \end{overpic}
\vspace{4mm}
    \caption{Resolving a crossing.}
    \label{surgery}
\end{figure}
\end{proof}

\subsection{Comparison with the grading by a noncommutative group}\label{sec:noncom}
We now compare our topological grading constructed in \S\ref{construction:G} with the gradings on $\mathcal{A}(\ZZ)$ defined in~\cite{lot}.

We first recall the definition of the noncommutative groups in which the gradings defined in~\cite{lot} takes values. The group $G'(4k)$ is a $\Z$-central extension of $H_1(Z\setminus z,\mathbf{a})$. In order to give a more concrete definition of $G'(4k)$, we need to recall a few definitions from~\cite{lot}. For a point $p\in\mathbf{a}$ and a Reeb chord $\sigma$, define 
$$m(p,\sigma)=\left\{\begin{array}{ll}1,&\text{ if }\sigma^-<p<\sigma^+,\\
1/2,&\text{ if }p=\sigma^-\text{ or }p=\sigma^+,\\
0,&\text{ otherwise}.
                             \end{array}\right.$$
We can extend $m$ bilinearly to a map $m:H_0(\mathbf{a})\times H_1(Z\setminus z,\mathbf{a})\to \frac{1}{2}\Z$. For $\alpha_1,\alpha_2\in H_1(Z\setminus z,\mathbf{a})$, one now defines $L(\alpha_1,\alpha_2)=m(\partial \alpha_1,\alpha_2)$, where $\partial:H_1(Z\setminus z,\mathbf{a})\to H_0(\mathbf{a})$ is the boundary map. Also, for $\alpha\in H_1(Z\setminus z,\mathbf{a})$, let $\epsilon(\alpha)$ be $1/4$ times the number of points $p$ in $\mathbf{a}$ such that the multiplicity of $\alpha$ on both sides of $p$ has different parity. Since the number of such points is always even, $\epsilon(\alpha)\in\frac{1}{2}\Z/\Z$. One can now define $$G'(4k)=\{(j,\alpha)\in \frac{1}{2}\Z\times H_1(Z\setminus z,\mathbf{a})\big\vert \epsilon(\alpha)\equiv j \,(\text{mod }1)\}.$$ The multiplication is defined by
$$(j_1,\alpha_1)\cdot (j_2,\alpha_2)=(j_1+j_2+L(\alpha_1,\alpha_2),\alpha_1+\alpha_2).$$
It follows from~\cite[Prop. 3.37]{lot} that this operation defines a multiplication in $G'(4k)$. For an element $g=(j,\alpha)\in G'(4k)$, the number $j\in\frac{1}{2}\Z$ is called the {\em Maslov component} of $g$, and $\alpha$ is called the {\em Spin$^c$ component} of $g$.

Given an element $a=(S,T,\phi)\in \A(4k)$, we consider the segments $[i,\phi(i)]$ seen as subsets of $Z\setminus z$. The sum of these segments determines a class $[a]\in H_1(Z\setminus z,\mathbf{a})$. We denote by $\text{inv}(a)$ the number of inversions of $a$. Let $\iota(a)=\text{inv}(a)-m(S,[a])$.
Then one defines
\begin{equation}\label{def:gr'}
 \gr'(a)=(\iota(a),[a])
\end{equation}

It follows from~\cite[Prop. 3.39]{lot} that $\gr'(a)\in G'(4k)$. Moreover, $\gr'$ is invariant under adding horizontal strands.
Let $\brho=\{\rho_1,\dots,\rho_n\}$ be a set of Reeb chords and $\s\subset[2k]$ be such that $I(\s)a(\brho)\neq 0$. We can see $\brho$ as an element of $\A(4k)$ with no horizontal strands. So $\gr'(I(\s)a(\brho))=\gr'(\brho)$. Let $|\brho|$ denote the number of elements of $\brho$, and let $|\text{ab}(\brho)|$ and $|\text{int}(\brho)|$ denote the number of abutting and interleaved pairs in $\brho$, respectively. Each chord contributes $-1/2$ to $m(S,[\brho])$. Moreover, for each nested or interleaved pair, we obtain an extra contribution of $-1$ to $m(S,[\brho])$. For each abutting pair, we obtain an extra contribution of $-1/2$ to $m(S,[\brho])$. We note that every inversion comes from a nested pair. So the contribution from the nested pairs is actually 0. Therefore
\begin{equation}\label{eq:i}
 \iota(\brho)=-\frac{|\brho|}{2}-\frac{|\text{ab}(\brho)|}{2}-|\text{int}(\brho)|.
\end{equation}

By making some non-canonical choices, one can also define a {\em refined grading} taking values in a subgroup of $G'(4k)$, see~\cite[\S 3.3.2]{lot}. That is necessary for the gluing theorems to behave well with respect to the grading. Alternatively, as suggested in~\cite[Rem. 10.44]{lot}, one could consider a more canonical subset of $G'(4k)$, as follows. Let $M_*:H_0(\mathbf{a})\to H_0([2k])$ denote the pushforward of the map $M:\mathbf{a}\to [2k]$. Define $G'(\ZZ)$ to be the set of elements $(j,\alpha)$ in $G'(4k)$ such that $M_*(\partial \alpha)=\mathbf{t}-\mathbf{s}$, for $\mathbf{t},\mathbf{s}\subset[2k]$, with $|\mathbf{t}|=|\mathbf{s}|$. We observe that $G'(\ZZ)$ is a groupoid and that $\gr'(a)\in G'(\ZZ)$ for every homogeneous element $a\in\A(\ZZ)$. We recall that the notation $G(\ZZ)$ was used in~\cite{lot} for the refined grading group, but in the current paper $G(\ZZ)$ denotes the groupoid on which the geometric grading takes values. We now have the following proposition.

\begin{prop}\label{prop:noncan}
 There exists a homomorphism $\mathcal{F}:G(\ZZ)\to G'(\ZZ)$ such that $\mathcal{F}(\text{\em gr}(a))=\text{\em gr}'(a)$ for every homogeneous element $a\in \A(\ZZ)$.
\end{prop}
\begin{proof}
Let $\mathfrak{N}$ and $\bar{\mathfrak{N}}$ be as in the proof of Proposition~\ref{prop:propgrad}. Let $\tau$ be the trivialization of $T(F\times[0,1])$ constructed in that proof.

For each $\s\subset[2k]$, we see $v_{\s}$ as a map $F\to S^2$. We can slightly perturb the vector fields $v_{\s}$ so that $(0,0,1)$ is a regular value of these maps. As in the proof of Proposition~\ref{prop:propgrad}, we observe that $v_{\s}^{-1}(0,0,1)=\s\cup P$, where $P$ is a set of points in the complement of $\bar{\mathfrak{N}}$ that does not depend on $\s$.

Now let $[v]\in G(\ZZ)$. Then $[v]\in G(\s,\t)$, for some $\s,\t\subset[2k]$, such that $|\s|=|\t|$. We see the vector field $v$ as a map $F\times[0,1]\to S^2$. We can slightly homotope $v$ in $F\times(0,1)$ so that $(0,0,1)$ is a regular value of $v$. Now consider $L_v:=v^{-1}(0,0,1)$. Observe that $L_v\cap (F\times\{0\} )=(\s\cup P)\times\{0\}$ and $L_v\cap (F\times\{1\})=(\t\cup P)\times\{1\}$. 
Since the map $\iota_*:H_1(\bar{\mathfrak{N}})\to H_1(F)$ induced from the inclusion is an isomorphism, it follows that $L_v$ is relatively framed homotopic to $\tilde{L}_v\cup (P\times [0,1])$, where $\tilde{L}_v$ is a framed one-manifold contained in $\bar{\mathfrak{N}}\times[0,1]$, which is transverse to $F\times\{t\}$ for all $t$, and the framing on $P\times [0,1]$ is trivial. By the Pontryagin-Thom construction, we can homotope $v$ and obtain $v'$ such that $L_{v'}=\tilde{L}_v\cup (P\times [0,1])$. So we can assume, without loss of generality, that $L_v=\tilde{L}_v\cup (P\times[0,1])$. Now let $K_v=v^{-1}(\delta,0,\sqrt{1-\delta^2})$, for a small $\delta>0$, such that $(\delta,0,\sqrt{1-\delta^2})$ is a regular value of $v$. We write $\tilde{K}_v=K_v\cap (\bar{\mathfrak{N}}\times[0,1])$. We will project $\tilde{L}_v$ and $\tilde{K}_v$ to $\bar{\mathfrak{N}}$ and for each intersection of the two projections, we keep track of which strand goes above which strand. Now we will make an extra assumption so that the count of intersections is well-defined. For each critical point $p_i\in[2k]$, recall that $H_i$ denotes the corresponding 1-handle and write $q_i=v_{\{p_i\}}^{-1}(\delta,0,\sqrt{1-\delta^2})$. Let $\mathcal{L}$ be the closure of a connected component of $L_v\cap((H_i\setminus \{p_i\})\times[0,1])$, let $t_0,t_1\in[0,1]$ be such that $\partial \mathcal{L}\subset (F\times\{t_0\})\cup (F\times\{t_1\})$ and let $\eta_{\mathcal{L}}=K_v\cap(H_i\times[t_0,t_1])$. We say that $\mathcal{L}$ is {\em standard} if:
\begin{itemize}
\item $L_v\cap (H_i\times[t_0,t_1])=\mathcal{L}$,
 \item the projection of $\mathcal{L}$ to $H_i$ is a one-manifold contained in the core of $H_i$,
\item  the projection of $\eta_{\mathcal{L}}$ to $H_i$ is a one-manifold whose boundary contains $q_i$,
\item the projection of the braid $(\mathcal{L},\eta_{\mathcal{L}})$ is isotopic to one of the four possibilities depicted in Figure~\ref{poss}.
\end{itemize}
We say that $v$ is {\em standard} if every connected component of $L_v\cap((H_i\setminus \{p_i\})\times[0,1])$ is standard, for all $i$. Up to a relative homotopy of $v$, we can assume $v$ is standard.

Let $L_1,\dots,L_n$ be the connected components of $\tilde{L}_v\cap(\mathfrak{N}\setminus\times[0,1])$. We observe that $\partial L_j$ is contained in a small neighborhood of $\a\times[0,1]$. So $L_j$ gives rise to an element in $H_1(Z\setminus z,\mathbf{a})$, which we denote by $[L_j]$. We define $$\mathcal{F}_{\text{sp}}([v])=\sum_{j=1}^n[L_j]\in H_1(Z\setminus z,\mathbf{a}).$$ Note that $M_*(\partial \mathcal{F}_{\text{sp}}([v]))=\t-\s$. 

Now we define the Maslov component $\mathcal{F}_{m}([v])$. Let $K_1,\dots,K_n$ be the connected components of $\tilde{K}_v\cap(\mathfrak{N}\times[0,1])$ corresponding to $L_1,\dots,L_n$. For $1\le j,l\le n$, let $L_j\cdot K_l$ denote the signed count of intersections of the projections of $L_j\cap((\mathfrak{N}\setminus \bigcup_iH_i)\times[0,1])$ and $K_l\cap((\mathfrak{N}\setminus\bigcup_i H_i)\times[0,1])$ to $\mathfrak{N}$. The signs corresponding to each intersection are determined by our sign convention, as in Figure~\ref{signs}. Note that this is the opposite of the usual sign convention. Now we define $$\mathcal{F}_m([v])=\frac{1}{2}\sum_{j,l=1}^nL_j\cdot K_l.$$

\begin{figure}[htb]
\centering \def\svgwidth{450pt}
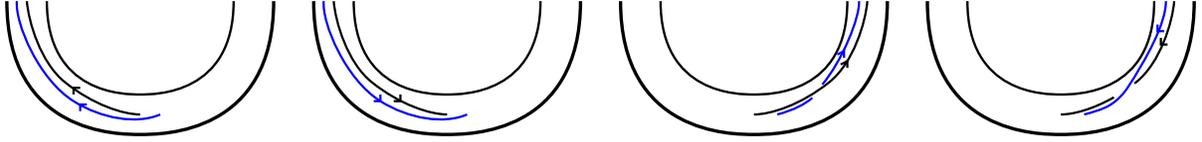 
\caption{The four possibilities of a standard $\mathcal{L}$. The framing is indicated by the blue arc.} \label{poss}
\end{figure}

We now claim that $\epsilon(\mathcal{F}_{\text{sp}}([v]))\equiv\mathcal{F}_m([v])$ (mod 1). We can compute $\epsilon(\mathcal{F}_{\text{sp}}([v]))$ as follows. For each $p\in\mathbf{a}$, we define $m(p)$ to be the number of classes $[L_j]$ whose boundary contains $p$. We observe that $\sum_{p\in\mathbf{a}} m(p)=2n$. Now let $\text{par}(m(p))=1$ if $m(p)$ is odd, and $\text{par}(m(p))=0$ if $m(p)$ is even. By definition, $\epsilon(\mathcal{F}_{\text{sp}}([v]))\equiv\frac{1}{4}\sum_{p\in\mathbf{a}}\text{par}(m(p))\,(\text{mod 1}).$ We observe that
$$\frac{\text{par}(m(p))-m(p)}{4}\equiv\frac{m(p)(m(p)-1)}{4}\,\text{(mod 1)}.$$
So it is enough to show that 
\begin{equation}\label{eq:fm}
\mathcal{F}_m([v])\equiv \frac{1}{2}\Big(n+\sum_{p\in\mathbf{a}}\frac{m(p)(m(p)-1)}{2}\Big)\,(\text{mod 1}).
 \end{equation}
We first observe that the projections of $L_j$ and $K_j$ intersect an odd number of times in $\mathfrak{N}\setminus\bigcup_i H_i$, for every $j$. Now, for each point $p\in\mathbf{a}$, if $m(p)>1$, we obtain intersection betweens the projections of $L_j$ and $K_l$ for $j\neq l$ as in Figure~\ref{mp}. In fact, we obtain $1+2+\dots+(m(p)-1)=\frac{1}{2}(m(p)(m(p)-1))$ intersections. All other intersections of $L_j$ and $K_l$ for $j\neq l$ correspond to intersections of the projections of $L_j$ and $L_l$ for $j\neq l$ and they come in pairs, see Figure~\ref{lj}. So we obtain~\eqref{eq:fm}. Hence we can define $$\mathcal{F}([v])=(\mathcal{F}_{m}([v]),\mathcal{F}_{\text{sp}}([v]))\in G'(4k).$$

\begin{figure}
 \centering
\begin{minipage}[b]{0.35\textwidth}
\def\svgwidth{\textwidth}
 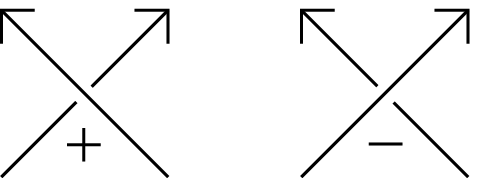
\subcaption{The signs of a crossing}
\label{signs}
\end{minipage}\qquad
\begin{minipage}[b]{0.25\textwidth}
\def\svgwidth{\textwidth}
 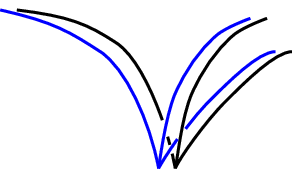
\subcaption{$p\in\mathbf{a}$ with $m(p)=3$}
\label{mp}
\end{minipage}\qquad
\begin{minipage}[b]{0.2\textwidth}
\def\svgwidth{\textwidth}
 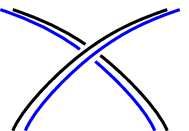
\subcaption{$L_j$ and $L_{j'}$}
\label{lj}
\end{minipage}
\caption{}
\end{figure}

To prove that $\mathcal{F}$ is a homomorphism, we need to show that $\mathcal{F}([v\cdot w])=\mathcal{F}([v])\cdot\mathcal{F}([w])$. We observe that $L_{v\cdot w}$ is the concatenation of $L_{v}$ with $L_w$. We denote the connected components of $L_{v}\cap(\mathfrak{N}\times[0,1])$ by $L_1^v,\dots,L_n^v$ and the connected components of $L_{w}\cap(\mathfrak{N}\times[0,1])$ by $L_1^w,\dots,L_m^w$. So the connected components of $L_{v\cdot w}\cap(\mathfrak{N}\times[0,1])$ are $L_1^v,\dots,L_n^v,L_1^w,\dots,L_m^w$. Note that it might be possible to homotope $v\cdot w$ so that $L_{v\cdot w}\cap(\mathfrak{N}\times[0,1])$ has fewer than $n+m$ connected components, but we will not do so. It follows from our description of the connected components of $L_{v\cdot w}\cap(\mathfrak{N}\times[0,1])$ that $$\mathcal{F}_{\text{sp}}([v+w])=\mathcal{F}_{\text{sp}}([v])+\mathcal{F}_{\text{sp}}([w]).$$ We define $K_1^v,\dots,K_n^v,K_1^w,\dots,K_m^w$ analogously for $K_v$ and $K_w$. We can choose $v$ and $w$ such that the projections to $\mathfrak{N}$ of all of these one-manifolds intersect transversely in $\mathfrak{N}\setminus \bigcup_i H_i$ and such that $v$ and $w$ are standard. In particular, $v\cdot w$ is standard. So 
$$\mathcal{F}_m([v\cdot w])=\frac{1}{2}\Big(\sum_{j,l=1}^nL_j^v\cdot K_l^v+\sum_{j=1}^n\sum_{l=1}^m(L_j^v\cdot K_l^w+ L_l^w\cdot K_j^v)+\sum_{j,l=1}^mL_j^w\cdot K_l^w\Big).$$ It follows from the definition that 
\begin{align*}
 \mathcal{F}_m([v])&=\frac{1}{2}\sum_{j,l=1}^nL_j^v\cdot K_l^v,\\
\mathcal{F}_m([w])&=\frac{1}{2}\sum_{j,l=1}^mL_j^w\cdot K_l^w.
\end{align*}
We now claim that \begin{equation}\label{eq:l}
\frac{1}{2}(L_j^v\cdot K_l^w+L_l^w\cdot K_j^v)=L([L_j^v],[L_l^w]).                                                                                                                                                                                                                                                                                                                                                                                                                                                                                                                                                                                                                                                                                  \end{equation}
We note that both sides of~\eqref{eq:l} change by a factor of $-1$ if we change the orientation of either $L_j^v$ or $L_l^w$. So we can assume that the projections of $L_j^v$ and $L_l^w$ are positively oriented with respect to $Z\setminus z$. We can also assume that the projections of $(L_j^v,K_j^v)$ and $(L_l^w,K_l^w)$ only intersect in $\mathfrak{N}$ near $\partial[L_j^v]$. Note that $(L_j^v,K_j^v)$ is always below $(L_l^w,K_l^w)$. Now we write $\partial[L_j^v]=p_1-p_0$. For $i=0,1$, if $p_i$ is in the interior of $[L_l^w]$, we obtain a contribution of $(-1)^i$ to both sides of~\eqref{eq:l}, and if $p_i$ is on the boundary of $[L_l^w]$, we obtain a contribution of $(-1)^i/2$ to both sides of~\eqref{eq:l}. For examples, see Figure~\ref{interleaved} and~\ref{abutting}, respectively. Hence~\eqref{eq:l} holds. Therefore $\mathcal{F}_m([v]\cdot[w])$ is the Maslov component of $\mathcal{F}([v])\cdot\mathcal{F}([w])$. We also observe that $\mathcal{F}([v]^{-1})=\mathcal{F}([v])^{-1}$.

It remains to show that for a generator $I(\s)a(\brho)$ of $\A(\ZZ)$, we have $\mathcal{F}(\gr(I(\s)a(\brho)))=\gr'(\brho)$. We first order $\rho_1,\dots,\rho_n$ as in Step 3 of \S\ref{construction:G}. Let $v$ be the vector field constructed in \S\ref{construction:G} whose relative homotopy class is $\gr(I(\s)a(\brho))$. Let $L_v$ and $K_v$ be as above. The 1-manifold $L_v\cap(\mathfrak{N}\times[0,1])$ is the union of arcs $L_i$, one for each Reeb chord $\rho_i$. Up to a relative isotopy of $L_v$, we can assume that the projection of $L_v\cap(\mathfrak{N}\times[0,1])$ has minimal number of intersections, i.e. there is no relative isotopy of $L_v$ that decreases the number of intersections. It follows from the ordering of the Reeb chords that if the projections of $L_i$ and $L_j$ intersect for $i<j$, then the pair $\{\rho_i,\rho_j\}$ is interleaved and this is a negative intersection. Now the framing on $L_i$ is trivial, so $(L_i,K_i)$ is as shown in Figure~\ref{reeb}. Thus $L_i\cdot K_i=-1$ So for each Reeb chord $\rho_i$, we get a contribution of $-1/2$ to the Maslov component of $\mathcal{F}(\gr(I(\s)a(\brho)))$. Moreover, each interleaved pair gives rise to two negative intersections of the projections of $L_v$ and $K_v$, see Figure~\ref{interleaved}. So each interleaved pair contributes $-1$ to the Maslov component of $\mathcal{F}(\gr(I(\s)a(\brho)))$.
Finally, if $\rho_i$ and $\rho_j$ abut, then we get an extra negative intersection, see Figure~\ref{abutting}. So an abutting pair contributes $-1/2$ to the Maslov component of $\mathcal{F}(\gr(I(\s)a(\brho)))$. Therefore, using \eqref{eq:i}, we conclude that
$$\mathcal{F}_m(\gr(I(\s)a(\brho)))=\iota(\brho).$$
Hence $\mathcal{F}(\gr(I(\s)a(\brho)))=\gr'(\brho)$.
\end{proof}

\begin{figure}
 \centering
\begin{minipage}[b]{0.22\textwidth}
\def\svgwidth{\textwidth}
 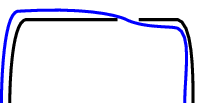
\subcaption{A Reeb chord}
\label{reeb}                        
\end{minipage}\qquad
\begin{minipage}[b]{0.4\textwidth}
\def\svgwidth{\textwidth}
 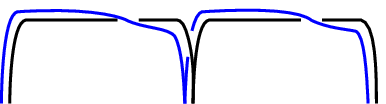
\subcaption{An abutting pair}
\label{abutting}
\end{minipage}\qquad
\begin{minipage}[b]{0.25\textwidth}
\def\svgwidth{\textwidth}
 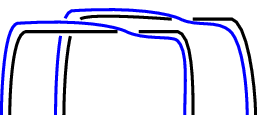
\subcaption{An interleaved pair}
\label{interleaved}
\end{minipage}
\caption{}
\end{figure}

\section{Grading on the modules}\label{sec:mod}

Let $Y$ be an oriented, connected, compact 3-manifold with connected boundary. Following \cite{lot}, we consider the bordered Heegaard diagram $$\mathcal{H} = (\Sigma,\alpha^c_1,\cdots,\alpha^c_{g-k},\alpha^a_1,\cdots,\alpha^a_{2k},\beta_1,\cdots,\beta_{g},z)$$ which is compatible with $Y$ in the sense that the following conditions are satisfied:

\begin{itemize}
    \item $\Sigma$ is a compact oriented surface with a single boundary component.
    \item $(\Sigma \cup_\partial D^2,\alpha^c,\beta)$ is a Heegaard diagram for $Y$.
    \item $\alpha^a_1,\cdots,\alpha^a_{2k}$ are pairwise disjoint, embedded arcs in $\Sigma$ with boundary on $\partial\Sigma$, and are disjoint from the $\alpha^c_i$.
    \item $\Sigma \setminus (\alpha^c_1 \cup \cdots \cup \alpha^c_{g-k} \cup \alpha^a_1 \cup \cdots \cup \alpha^a_{2k})$ is a disk with $2(g-k)$ holes.
    \item $z$ is a point in $\partial\Sigma$, disjoint from all of the $\alpha^a_i$.
\end{itemize}

We will abbreviate $\bm{\alpha}^c = \alpha^c_1 \cup \cdots \cup \alpha^c_{g-k}$, $\bm{\alpha}^a = \alpha^a_1 \cup \cdots \cup \alpha^a_{2k}$, $\bm\alpha = \bm{\alpha}^c \cup \bm{\alpha}^a$, and $\bm\beta = \beta_1 \cup \cdots \cup \beta_{g}$.

In this section, we explain how to define the grading on the modules $\widehat{\mathit{CFA}}(\H)$ and $\widehat{\mathit{CFD}}(\H)$. We start by defining the grading sets $S(\H)$ and $\bar{S}(\H)$.

\subsection{The grading set}\label{grading:set}

Let $F=\partial Y$. We recall from \cite{lot} that $\H$ gives rise to a pointed matched circle $\ZZ=(Z,\a,M,z)$, where $Z=\partial \Sigma$, $\a=\bm\alpha^c\cap Z$ and $M$ maps both points in $\alpha^c_i\cap Z$ to $i\in[2k]$ for every $i$. For $\s\in[2k]$, we denote by $\Vect(Y,v_{\s})$ the set of homotopy classes of nonvanishing vector fields in $Y$ whose restriction to $F$ is $v_{\s}$. Since $F$ is connected, $\Vect(Y,v_{\s})$ is nonempty if and only if $|\s|=k$. Let $$S(\H)=\coprod_{|\s|=k}\Vect(Y,v_{\s}).$$

We observe that the groupoid $G(\ZZ)$ acts on $S(\H)$ on the right by concatenation. More precisely, given vector fields $v$ and $w$ such that $[v]\in\Vect(Y,v_{\s})$ and
$[w]\in G(\s,\t)$, define $[v]\cdot [w]$ as follows. Identify a collar neighborhood $N(F)$ of $F$ in $Y$ with $F\times [0,1]$ and take a representative $\tilde{v}$ of $[v]$ which is $[0,1]$-invariant in $N(F)\cong F\times [0,1]$. Now define $[v]\cdot[w]\in\Vect(Y,v_{\t})$ to be the relative homotopy class of the vector field which equals $\tilde{v}$ in the complement of $N(F)$ and $w$ in $N(F)\cong F\times [0,1]$. 
Note that we also have a $\Z$-action on $S(\H)$ just as before, which we again denote multiplicatively by $\lambda^n$ on the left.
We also observe that this action need not be free. In fact, let $[v]\in S(\H)$ and denote by $v^{\perp}$ the orthogonal complement of $v$, seen as a complex line bundle. Then $\lambda^d\cdot[v]=[v]$ for every $d= \langle c_1(v^{\perp}),A\rangle$, for some $A\in H_2(Y)$.

Now we denote by $-\ZZ$ the pointed matched circle obtained by switching the orientation of $Z$, i.e. $-\ZZ=(-Z,\a,M,z)$.
We observe that the groupoid $G(-\ZZ)$ acts on $S(\H)$ on the left, as follows. Given a vector field $w$ in $(-F)\times[0,1]$, we define $\bar{w}$ to be the vector field in $F\times[0,1]$ given by $\bar{w}(x,t)=w(x,1-t)$. So, given a vector field $v$ in $Y$, if $v$ and $\bar{w}$ coincide along $F\cong F\times\{1\}$, we can glue them along $F\cong F\times\{1\}$ and obtain a new vector field in $Y$, which we denote by $\bar{w}\cdot v$. So, given $[w]\in G(\s,\t)\subset G(-\ZZ)$ and $[v]\in\Vect(Y,v_{\s})$, we can define $[w]\cdot[v]$ to be $[\bar{w}\cdot v]$.

The homotopy classes $[v],[w]\in\Vect(Y,v_{\s})$ are said to be in the same relative Spin$^c$ structure if $v$ is homotopic to $w$ on the 2-skeleton relative to the boundary.
We observe that there exists $n\in\Z$ such that $[v]=\lambda^n\cdot [w]$ if, and only if, $[v],[w]\in\Vect(Y,v_{\s})$ and $v$ an $w$ are in the same relative Spin$^c$ structure.

\subsection{Homotopy classes of vector fields}\label{sec:htpy}

The goal of this section is provide a new way to compute the difference between homotopy classes of nonvanishing vector fields, based on the Pontryagin-Thom construction. The construction here is inspired by and very similar to the work of Dufraine~\cite{du}. Let $Y$ be a closed oriented 3-manifold. Suppose $\xi,\eta$ are nonvanishing vector fields on $Y$. By a $C^\infty$-small perturbation, we can assume that the set $$L=L_{\xi,\eta}=\{y \in Y~|~\xi(y)=-\eta(y)\}$$ is a link in $Y$. In the case that $[L]=0 \in H_1(Y;\mathbb{Z})$, there exists an embedded compact surface $\Sigma \subset Y$ with $\partial\Sigma = L$. Choosing a Riemannian metric on $Y$, we consider the orthogonal complement $\eta^\bot$ of $\eta$, which is a co-oriented plane field on $Y$. Since $\Sigma$ deformation retracts onto a wedge of circles, we can choose a trivialization $\tau: \eta^\bot|_\Sigma \to \Sigma\times\mathbb{R}^2$. This in turn gives a trivialization $\tilde\tau: TY|_\Sigma \to \Sigma\times\mathbb{R}^3$ by setting $\tilde\tau^*(\partial_z)$ to be equal to $\eta$, where $(x,y,z)$ are the coordinates in $\mathbb{R}^3$. Let $N(\Sigma)$ denote a small tubular neighborhood of $\Sigma$ in $Y$. Then $\tilde{\tau}$ gives rise to a trivialization $TY|_{N(\Sigma)}\cong N(\Sigma)\times\R^3$.

Using the above trivialization, we can see $\xi|_{N(\Sigma)}$ as a map $\xi_ \tau:N(\Sigma) \to S^2 \subset \mathbb{R}^3$. It is clear from the construction that $L_{\xi,\eta}=\xi_ \tau^{-1}(0,0,-1)=\partial \Sigma$. Taking the pre-image of a regular value close to $(0,0,-1)$ in $S^2$, we get a framing on $L_{\xi,\eta}$. We represent this framing by a number $n_{\xi,\eta}$, given by the difference from the Seifert framing. We note that $n_{\xi,\eta}$ is independent of the Seifert surface and the trivialization of $\eta^\bot|_\Sigma$, modulo the divisibility of $c_1(\eta^\bot)$. The following proposition gives a way to compute the difference between homotopy classes of nonvanishing vector fields. The result was essentially known by Dufraine~\cite{du} but we write down a proof here for the readers' convenience.

\begin{prop} \label{htpydiff}
Let $\xi$ and $\eta$ be vector fields on $Y$ and let $d$ denote the divisibily of $c_1(\eta^\bot)$. Then
\begin{enumerate}[(a)]
\item $\xi$ is homotopic to $\eta$ if and only if $L_{\xi,\eta}$ is nullhomologous and $n_{\xi,\eta}\equiv0$ (mod $d$).
 \item $\xi$ and $\eta$ are in the same Spin$^c$ structure if and only if $L_{\xi,\eta}$ is nullhomologous. If that is the case, then $[\xi]=\lambda^{n_{\xi,\eta}}\cdot[\eta]$.
\end{enumerate}
\end{prop}

\begin{proof}
We start by proving (a). Suppose there exists a 1-parameter family of nonvanishing vector fields $\{\xi_t\}_{t\in[0,1]}$ on $Y$ such that $\xi_0=\xi$, $\xi_1=\eta$. We choose a Riemannian metric on $Y$ such that $\xi_t$ is of unit length. Therefore we define a section $\Xi: Y\times[0,1] \to STY\times[0,1]$ by $\Xi(y,t)=(\xi_t(y),t)$ for all $y \in Y, t \in [0,1]$, where $STY$ denotes the unit tangent bundle. We can also define a section $\mathbb{I}: Y\times[0,1] \to STY\times[0,1]$ by $\mathbb{I}(y,t)=(-\eta(x),t)$.
We observe that $L_{\xi,\eta}=\{(y,0) \in Y\times[0,1]~|~\Xi(y,0)=\mathbb{I}(y,0)\}$ and $\{(y,1) \in Y\times[0,1]~|~\Xi(y,1)=\mathbb{I}(y,1)\}=\emptyset$. By the standard transversality argument, we can assume that $$\{(y,t) \in Y\times[0,1]~|~\Xi(y,t)=\mathbb{I}(y,t)\}$$ is an embedded surface in $Y\times[0,1]$. Therefore $[L_{\xi,\eta}]=0 \in H_1(Y;\mathbb{Z})$.

Conversely, let $\Sigma \subset Y$ be a compact surface such that $\partial\Sigma = L_{\xi,\eta}$, and consider a neighborhood $N(\Sigma)$ of $\Sigma$ in $Y$. Observe that $\xi$ is homotopic to $\eta$ on the complement of $N(\Sigma)$ by a linear homotopy, so we can assume that $\xi=\eta$ on $Y \setminus N(\Sigma)$. Since, again, $N(\Sigma)$ deformation retracts onto a wedge of circles, we can trivialize $\eta^\bot|_{N(\Sigma)}$ and therefore obtain a trivialization of $TY|_{N(\Sigma)}$ by writing $TY=\eta\oplus\eta^{\bot}$. The vector field $\xi$, under this trivialization, sends $L_{\xi,\eta}$ to $(0,0,-1) \in S^2$ as before. The Pontryagin-Thom construction asserts that $\xi$ is homotopic to $\eta$ if and only if the link $L_{\xi,\eta}$ with framing $n_{\xi,\eta}$ is framed cobordant to the empty set. This happens if and only if $L_{\xi,\eta}$ is nullhomologous and $n_{\xi,\eta}\equiv0$ (mod $d$).

We now prove (b). If $\xi$ and $\eta$ are in the same Spin$^c$ structure, then there exists $m\in\Z$ such that $[\xi]=\lambda^m\cdot[\eta]$. Let $\tilde\eta$ be a nonvanishing vector field in $Y$ given by modifying $\eta$ in a very small ball, corresponding to the action of $\lambda^m\in\pi_3(S^2)$. By definition, $[\tilde\eta]=\lambda^m\cdot[\eta]$. So $\xi$ and $\tilde\eta$ are homotopic. By (a), $[L_{\xi,\tilde\eta}]=0$. Moreover, $L_{\xi,\tilde\eta}$ is obtained from $L_{\xi,\eta}$ by a link contained in a ball. Therefore $L_{\xi,\eta}$ is also nullhomologous.

Conversely if $[L_{\xi,\eta}]=0$, then, as explained above, we obtain a framing $n_{\xi,\eta}$ on $L_{\xi,\eta}$. Now we act on $\eta$ by $\lambda^{n_{\xi,\eta}}\in\pi_3(S^2)$, obtaining a vector field $\tilde\eta$. We observe that $L_{\xi,\tilde{\eta}}$ is still nullhomologous and that $n_{\xi,\tilde\eta}=0$. By (a), we conclude that $[\xi]=[\tilde\eta]$. So $[\xi]=\lambda^{n_{\xi,\eta}}\cdot[\eta]$. In particular, $\xi$ and $\eta$ are in the same Spin$^c$ structure. Note that we also proved the second assertion.
\end{proof}

\begin{rmk}
The point of our approach is that in order to compute the difference between $\xi$ and $\eta$, it suffices to trivialize $TY$ along a Seifert surface, which is much easier in practice.
\end{rmk}

\subsection{Grading on $\widehat{\mathit{CFA}}(\H)$}\label{gr:cfa}
We start by recalling the definition of the $\mathcal{A}_{\infty}$-module $\widehat{\mathit{CFA}}(\H)$ from~\cite{lot}.
Let $\mathfrak{S}(\mathcal{H})$ be the set of $g$-tuples $\mathbf{x}=\{x_1,\cdots,x_g\} \subset \bm\alpha \cap \bm\beta$ such that there is exactly one point $x_i$ on each $\beta$-circle and on each $\alpha$-circle and there is at most one $x_i$ on each $\alpha$-arc. Then $\widehat{\mathit{CFA}}(\H)$ is generated as a vector space over $\Z/2$ by $\mathfrak{S}(\mathcal{H})$. We also recall that given $\xx\in\mathfrak{S}(\mathcal{H})$, there is an idempotent $I_A(\xx):=I(o(\xx))$, where $o(\xx)\subset[2k]$ is the set of $\alpha$-arcs containing $x_i$ for some $i$. We have a right action of the ring of idempotents $\mathcal{I}:=\mathcal{I}(\ZZ)$ on $\widehat{\mathit{CFA}}(\H)$ given by
\begin{equation*}
 \xx\cdot I(\s)=\left\{\begin{array}{ll}
                 \xx,&\text{if }I_A(\xx)=I(\s),\\
		 0,&\text{otherwise}.
                \end{array}\right.
\end{equation*}
Let $\A:=\A(\ZZ)$. As explained in~\cite[Ch. 7]{lot}, the $\mathcal{A}_{\infty}$-structure on $\widehat{\mathit{CFA}}(\H)$ is given by maps
$$m_{l+1}:\widehat{\mathit{CFA}}(\H)\otimes_{\mathcal{I}}\A\otimes_{\mathcal{I}}\dots\otimes_{\mathcal{I}}\A\to\widehat{\mathit{CFA}}(\H).$$

Now we want to define a grading function $$\textrm{gr}: \mathfrak{S}(\mathcal{H}) \to S(\H),$$ compatible with the maps $m_{l+1}$. More precisely, let $\xx\in\mathfrak{S}(\mathcal{H})$ and let $a(\brho_1),\dots,a(\brho_l)$ be generators of $\A$. If 
\begin{equation}
\xx\otimes_{\mathcal{I}}a(\brho_1)\otimes_{\mathcal{I}}\dots\otimes_{\mathcal{I}}a(\brho_l)\neq 0\label{eq:cfa}
\end{equation}
then we can write
$$\xx\otimes_{\mathcal{I}}a(\brho_1)\otimes_{\mathcal{I}}\dots\otimes_{\mathcal{I}}a(\brho_l)=\xx\otimes_{\mathcal{I}}I(\s_1)a(\brho_1)\otimes_{\mathcal{I}}\dots\otimes_{\mathcal{I}}I(\s_l)a(\brho_l),$$
for some $\s_1,\dots,\s_l\subset[2k]$.
Note, in particular, that $I(\s_1)=I_A(\xx)$. If $\yy$ is a summand in $m_{l+1}(\xx,a(\brho_1),\dots,a(\brho_l))$, we want $\gr$ to satisfy
$$\gr(\yy)=\lambda^{l-1}\cdot \gr(\xx)\cdot \gr(I(\s_1)a(\brho_1))\dots \gr(I(\s_l)a(\brho_l)).$$

Recall the following definition from~\cite[Definition 4.8]{lot}.
\begin{definition}
Given a compact 3-manifold $Y$ with bordered Heegaard diagram $\mathcal{H}$, we say that a pair consisting of a Riemannian metric $g$ on $Y$ and a self-indexing Morse function $h: Y \to [0,3]$ is {\em compatible with $\mathcal{H}$} if
\begin{itemize}
    \item the boundary of $Y$ is geodesic,
    \item the gradient vector field $\nabla h|_{\partial Y}$ is tangent to $\partial Y$,
    \item $h$ has a unique index 0 and a unique index 3 critical point, both of which lie on $\partial Y$, and are the unique index 0 and 2 critical points of $h|_{\partial Y}$, respectively,
    \item the index 1 critical points of $h|_{\partial Y}$ are also index 1 critical points of $h$,
    \item $h|_{\partial Y}$, viewed as a Morse function on $F=\partial Y$, is compatible with the pointed matched circle $\mathcal{Z}$.
\end{itemize}
\end{definition}

Fix a compatible Morse function $h:Y \to [0,3]$, and consider the gradient vector field $\nabla h$ on $Y$. For any $\mathbf{x} \in \mathfrak{S}(\mathcal{H})$, the pair $(\mathbf{x},z)$ determines $g+1$ gradient trajectories $\{\gamma_0,\cdots,\gamma_g\}$, where $\gamma_0$ connects the index 0 and index 3 critical points passing through $z$, and $\gamma_i$ connects the index 1 and index 2 critical points passing through $x_i$. We define $\textrm{gr}(\mathbf{x}) \in S(\H)$ by modifying $\nabla h$ near tubular neighborhoods of the trajectories $\gamma_i$ as follows.

Let $N(\gamma_0)$ be a small neighborhood of $\gamma_0$ in $Y$ and let $D=\{(x,y)\in\R^2|x^2+y^2\le 1,x\ge 0\}$. Then $N(\gamma_0)$ is diffeomorphic to $D\times[0,\pi]/\sim$, where the equivalence relation is given by $((0,y),t)\sim((0,y),t')$ for every $t,t'$, and where $(D\times\{0\})\cup (D\times\{\pi\})/\sim$ is identified with $N(\gamma_0)\cap \partial Y$, see Figure~\ref{nbhd2}(a). Using the above identification, the vector field $\nabla h$ restricted to $D\times\{t\}$ is depicted in Figure~\ref{nbhd}(a). For each $t\in[0,\pi]$, we modify $\nabla h$ in $D\times\{t\}$ as shown in Figure~\ref{nbhd}(d). Since these modifications coincide on $D\cap \{y=0\}$, we get a nonvanishing vector field on $D\times[0,\pi]/\sim$. This is the restriction to the half-ball of the analogous modification used to define the grading on Heegaard Floer homology~\cite{gh}. For a formula describing this modification, see~\cite[\S2]{gh}.

We order the flow lines $\gamma_1,\dots,\gamma_g$ so that the index one critical points corresponding to $\gamma_1,\dots,\gamma_k$ lie on $\partial Y$. For each $i=1,\dots,k$, let $N(\gamma_i)$ be a small neighborhood of $\gamma_i$ in $Y$. Let $\tilde{B}$ be the intersection of the unit ball in $\R^3$ with $\{z\ge -1/2\}$. Then $N(\gamma_i)$ is diffeomorphic to $\tilde{B}$, see Figure~\ref{nbhd2}(b). Let $\tilde{D}=\{(x,y)\in\R^2|x^2+y^2\le 1,y\ge -1/2\}$. Each vertical cross-section of $\tilde{B}$ can be identified with $\tilde{D}$. The vector field $\nabla h$ restricted to $N(\gamma_i)$ can be viewed as an interpolation between $\nabla h$ restricted to two transverse vertical cross-sections, corresponding to the unstable manifold of the index one critical point and the stable manifold of the index two critical point. Figure~\ref{nbhd}(b,c) shows the restriction of $\nabla h$ to these two cross-sections. We modify $\nabla h$ on these cross-sections as in Figure~\ref{nbhd}(e,f). Again, this is very similar to the corresponding construction on Heegaard Floer homology. Namely, this is the restriction to $\{z\ge -1/2\}$ of the vector field defined in~\cite{gh}. The reader can find a formula describing this modification in~\cite[\S2]{gh}.
For each $i=k+1,\dots,g$, the corresponding index one critical point lies in the interior of $Y$. So do the same modification as in~\cite[\S2]{gh}.

We still have to eliminate the boundary index one critical points which do not belong to any $\gamma_i$. We do so by slightly perturbing $\nabla h$ in a neighborhood of each of these points so that it points to the interior of $Y$. Alternatively, we observe that $Y$ is diffeomorphic to the complement of the union of small neighborhoods of each of these points. So $\nabla h$ restricted to a tubular neighborhood of the boundary of this complement gives the desired modification of $\nabla h$, see Figure~\ref{nbhd2}(c). Let $v_{\xx}$ denote the vector field in $Y$ obtained by modifying $\nabla h$ as explained above. Then we define $\text{gr}(\xx)$ to be the relative homotopy class of $v_{\xx}$. We note that $\text{gr}(\xx)\in\Vect(Y,v_{o(\xx)})$.

\begin{figure}[h]
    \begin{overpic}[scale=.3]{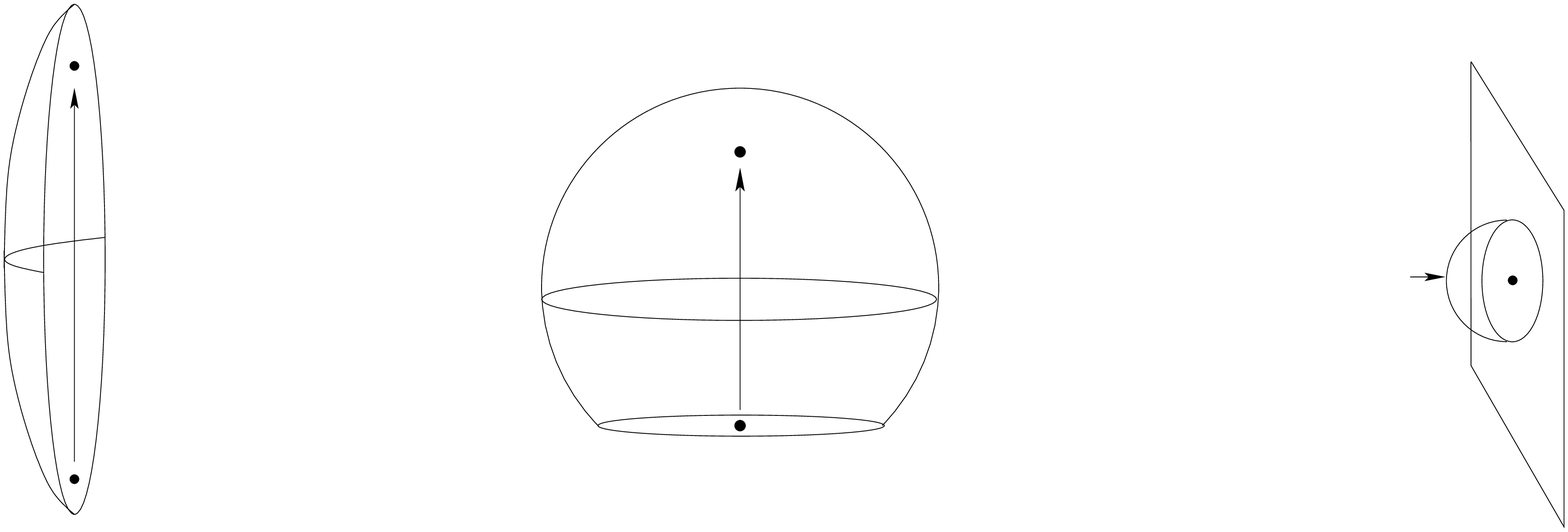}
    \put(4,-3){(a)}
    \put(45,-3){(b)}
    \put(93,-3){(c)}
    \end{overpic}
    \caption{}
    \label{nbhd2}
\end{figure}

\begin{figure}[h]
    \begin{overpic}[scale=.3]{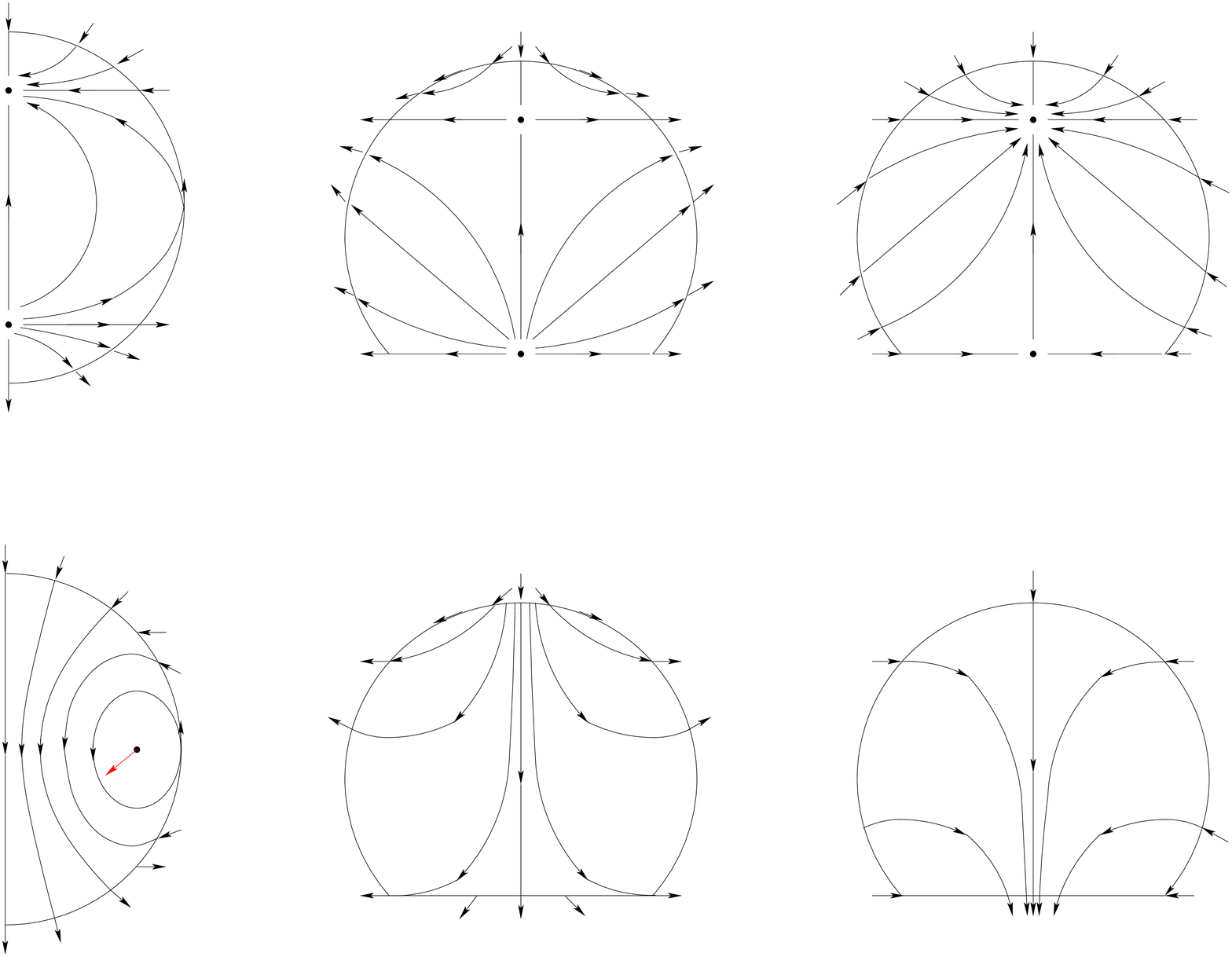}
    \put(4,41){(a)}
    \put(40,41){(b)}
    \put(82,41){(c)}
    \put(4,-3){(d)}
    \put(40,-3){(e)}
    \put(82,-3){(f)}
    \end{overpic}
    \vspace{2mm}
    \caption{Modifying $\nabla h$ to a nonvanishing vector field.}
    \label{nbhd}
\end{figure}

Following~\cite[Definition 4.14]{lot}, given generators $\mathbf{x},\mathbf{y} \in \mathfrak{S}(\mathcal{H})$, we consider the relative homology group
$$H_2(\Sigma\times[0,1]\times[0,1],((S_{\bm\alpha}\cup S_{\bm\beta} \cup S_{\partial})\times[0,1])\cup (G_{\xx}\times\{0\})\cup (G_{\yy}\times\{1\})),$$
where $S_{\bm\alpha}=\bm\alpha\times\{1\}$, $S_{\bm\beta}=\bm\beta\times\{0\}$, $S_{\partial}=(\partial\Sigma\setminus z)\times[0,1]$, $G_{\xx}=\xx\times[0,1]$ and $G_{\yy}=\yy\times[0,1]$. This group is usually denoted by $\pi_2(\xx,\yy)$, following the tradition from~\cite{osz}.

A homology class $B\in\pi_2(\xx,\yy)$ can be interpreted as a domain in $\Sigma$. As such, one defines $e(B)$ to be the Euler measure of this domain as follows. For each region in $\Sigma\setminus(\bm\alpha\cup\bm\beta)$, we define its Euler measure to equal its Euler characteristic $\chi(B)$ plus one quarter of the number of concave corners minus the number of convex corners. We can extend this linearly to domains in $\Sigma$. One also defines $n_{\xx}(B)$ to be one quarter of the number of components of $\Sigma\setminus(\bm\alpha\cup\bm\beta)$ in $B$ adjacent to $\xx$, counted with multiplicity. One defines $n_{\yy}$ similarly.
For $B\in\pi_2(\xx,\yy)$, one defines $\partial^{\partial} B$ to be the piece of the boundary of  $B$ contained in $\partial\Sigma$. We think of $\partial^{\partial} B$ as a class in $H_1(Z\setminus \{z\}, \a)$.
Let $\vec{\bm\rho}=(\brho_1,\dots,\brho_l)$ be an $l$-tuple of sets of Reeb chords. Let $L(\bm\rho_i,\bm\rho_j)$ denote the sum of all terms of the form $L(\rho,\sigma)$ for $\rho\in\bm\rho_i$ and $\sigma\in\bm\rho_j$. Recall that 
\begin{equation}\label{eq:iota}
\iota(\vec{\bm\rho})=\sum_{i=1}^l \iota(\bm\rho_i)+\sum_{i<j} L(\bm\rho_i,\bm\rho_j).\end{equation}
\begin{rmk}\label{rmk:iota}
 We recall that $\iota(\vec{\bm\rho})$ is the Maslov component of the product $\gr'(\brho_1)\cdots\gr'(\brho_l)$, where $\gr'$ is the noncanonical grading as in~\S\ref{sec:noncom}.
\end{rmk}
One can also define $[\vec{\bm\rho}]=[\brho_1]+\dots+[\brho_l]\in H_1(Z\setminus z,\a)$. Now recall the definition of ind$(B,\vec{\bm\rho})$ for $B\in\pi(\xx,\yy)$ and $\vec{\bm\rho}$ satisfying $\partial^{\partial} B=[\vec{\bm\rho}]$:
\begin{equation}
\text{ind}(B,\vec{\bm\rho})=e(B)+n_{\xx}(B)+n_{\yy}(B)+\iota(\vec{\bm\rho})+l.\label{eq:defind}
\end{equation}

Given $\mathbf{x},\mathbf{y} \in \mathfrak{S}(\mathcal{H})$ such that $\pi_2(\xx,\yy)$ is nonempty\footnote{This is equivalent to $\xx$ and $\yy$ being in the same Spin$^c$ structure.}, we now compare $\textrm{gr}(\mathbf{x})$ and $\textrm{gr}(\mathbf{y})$. The main result of this section is the following proposition.

\begin{prop} \label{index}
Let $\mathbf{x},\mathbf{y} \in \mathfrak{S}(\mathcal{H})$, $B\in\pi_2(\xx,\yy)$ and $\vec{\brho}=(\brho_1,\dots,\brho_l)$ such that $\partial^{\partial} B=[\vec{\bm\rho}]$. Assume that $(\xx,\vec{\bm\rho})$ satisfies~\eqref{eq:cfa} and let $\s_1,\dots,\s_l\subset[2k]$ such that
$$\xx\otimes_{\mathcal{I}}a(\brho_1)\otimes_{\mathcal{I}}\cdots\otimes_{\mathcal{I}}a(\brho_l)=\xx\otimes_{\mathcal{I}}I(\s_1)a(\brho_1)\otimes_{\mathcal{I}}\dots\otimes_{\mathcal{I}}I(\s_l)a(\brho_l).$$ Then \begin{equation}\label{grading:eq}\text{\em gr}(\mathbf{x}) \cdot \text{\em gr}(I(\s_1)a(\brho_1))\cdots\text{\em gr}(I(\s_l)a(\brho_l))=\lambda^{\text{\em ind}(B,\vec{\bm\rho})-l} \cdot \text{\em gr}(\mathbf{y}).\end{equation}
\end{prop}

\begin{proof}
We divide the proof in five steps.

\textsc{Step 1}: We start by making a simplifying assumption.

For each $\brho_i$, we write $\brho_i=\{\rho_{i,1},\dots,\rho_{i,|\brho_i|}\}$, where the Reeb chords as ordered as in Section~\ref{construction:G}. Now let
$\vec{\rho}=(\rho_{1,1},\dots,\rho_{1,|\brho_1|},\dots,\rho_{l,1},\dots,\rho_{l,|\brho_l|})$ and let $q=|\brho_1|+\dots+|\brho_l|$. It follows from~\eqref{eq:i} and~\eqref{eq:iota} that $\iota(\vec{\brho})=\iota(\vec{\rho})$. Write $\vec{\rho}=(\rho_1,\dots,\rho_q)$. We now assume that the Reeb chords $\rho_i$ are pairwise disjoint and that
\begin{equation}
 \Big\vert([2k]\setminus o(\yy))\cap\bigcup_{i=1}^q \{M(\rho_i^-)\}\Big\vert=\Big\vert([2k]\setminus o(\xx))\cap\bigcup_{i=1}^q \{M(\rho_i^+)\}\Big\vert=q.\label{eq:cond}
\end{equation}
This condition means that:
\begin{itemize}
 \item For every $p\in[2k]$, there is at most one $i$ such that $M(\rho_i^-)=p$, and there is at most one $j$ such that $M(\rho_j^+)=p$.
 \item For every $p\in o(\xx)$, there is no $i$ such that $M(\rho_i^+)=p$.
 \item For every $p\in o(\yy)$, there is no $i$ such that $M(\rho_i^-)=p$.
\end{itemize}
We will first prove the proposition under this assumption.

Since the Reeb chords $\rho_i$ are pairwise disjoint, $\iota(\vec{\rho})=-q/2$.
It follows from~\eqref{eq:defind} that
\begin{equation}
 \text{ind}(B,\vec{\brho})-l=e(B)+n_{\xx}(B)+n_{\yy}(B)-\frac{q}{2}.\label{eq:B2}
\end{equation}

We also note that~\eqref{eq:cond} implies that $a(\brho_i)=a(\rho_{i,1})\cdots a(\rho_{i,|\brho_i|})$, for all $i$. It follows that
$$\gr(I(\s_1)a(\brho_1))\cdots\gr(I(\s_l)a(\brho_l))=\gr(I(s_1) a(\rho_1))\cdots\gr(I(s_q) a(\rho_q)),$$
for some $s_1,\dots,s_q\subset [2k]$.

\textsc{Step 2}: We now use $(\xx,\vec{\rho})$ and $\yy$ to construct Heegaard Floer homology generators of a closed three-manifold related to $Y$.

Let $\Sigma'$ be a closed surface obtained by gluing a compact surface of genus $k$ with boundary $-Z$ to $\Sigma$ along the boundary. We construct a Heegaard diagram $(\Sigma',\bm\alpha',\bm\beta',z)$ as follows. For each arc $\alpha_i^a$, we glue an arc on $\Sigma'\setminus\Sigma$ to obtain a closed circle on $\Sigma'$, which we denote by $\alpha_i'$. We can always choose the completion of the $\alpha$-arcs such that $\bm\alpha'=\{\alpha_1^c,\dots,\alpha_{g-k}^c,\alpha_1',\dots,\alpha_{2k}'\}$ is a set of pairwise disjoint curves which are linearly independent in $H_1(\Sigma')$.
Recall that $Z\setminus N(z)\subset \partial \Sigma$ is a line segment containing all Reeb chords. Now consider $k$ translates of $Z\setminus N(z)$ in a collar neighborhood of $\partial \Sigma$ on $\Sigma'\setminus\Sigma$ ordered by their distance to $\bdry\Sigma$. For each $i=1,\dots,k$, we define a circle $\beta_i'$ on $\Sigma'\setminus\Sigma$ containing the $i$-th translate of $Z\setminus N(z)$, such that these circles are pairwise disjoint and linearly independent in homology. So we let $\bm\beta'=\{\beta_1,\dots,\beta_g,\beta_1',\dots,\beta_k'\}$. Therefore we obtain a Heegaard diagram $(\Sigma',\bm\alpha',\bm\beta',z)$, which gives rise to a closed three-manifold containing $Y$, denoted by $Y'$. We note that this diagram is similar but not identical to the diagram $\mathsf{AZ}(\ZZ)$, see~\cite{lot_mor}.

The domain $B \in \pi_2(\mathbf{x},\mathbf{y})$ naturally extends over $\Sigma'$, as follows. Note that, by~\eqref{eq:cond}, $q\le k$. Now each Reeb chord $\rho_i$ can be translated to $\beta_i'$ giving rise to a segment, whose endpoints are on the $\alpha$-circles corresponding to the endpoints of $\rho_i$. So each $\rho_i$ gives rise to two intersection points on $\beta_i'$. We add new intersection points to $\xx$ and $\yy$, as follows. For each $\rho_i^-$, the corresponding intersection point on $\beta_i'$ is added to $\mathbf{y}$ and, for each $\rho_i^+$, it is added to $\mathbf{x}$. If $q<k$, for each $i>q$, we choose a fixed intersection point on $\beta_i'$ to add to both $\mathbf{x}$ and $\mathbf{y}$. This construction gives rise to elements $\mathbf{x}'$ and $\mathbf{y}'$ of $\TT_{\bm\alpha'} \cap \TT_{\bm\beta'}$. We obtain a domain $B'$ on $\Sigma'$ by taking the union of $B$ with a domain in $\Sigma'\setminus\Sigma$ bounded by the Reeb chords $\rho_i$, its translates and the corresponding $\alpha$-circles. We observe that $B'\in\pi_2(\xx',\yy')$.

In~\cite{gh}, we defined an absolute grading function $\widetilde{\textrm{gr}}: \mathbb{T}_{\alpha'} \cap \mathbb{T}_{\beta'} \to \Vect(Y').$ This function is such that \begin{equation}\label{eq:grcf}\widetilde{\textrm{gr}}(\mathbf{x}')=\lambda^{\textrm{ind}(B')}\cdot \widetilde{\textrm{gr}}(\mathbf{y}'),\end{equation}
where $\text{ind}(B')$ is given by Lipshitz's index formula~\cite{lip}:\begin{equation}\label{combind}\textrm{ind}(B') = e(B')+n_{\xx'}(B')+n_{\yy'}(B').\end{equation}

We observe that $e(B)=e(B').$ The points in $\xx'$ and $\yy'$ are either elements of $\xx$ and $\yy$ or new corners on $\Sigma'\setminus\Sigma$, unless they are belong to $\beta_i'$ for $i>l$, in which case they do not contribute to $n_{\xx'}(B')$ and $n_{\yy'}(B')$. We have $2q$ new corners, giving a contribution of $q/2$ to $n_{\xx'}(B')+n_{\yy'}(B')$. Hence, it follows from~\eqref{eq:B2} and~\eqref{combind} that
\begin{equation}
 \text{ind}(B,\vec{\brho})-l=\textrm{ind}(B')-q.\label{eq:B'}
\end{equation}

\textsc{Step 3}: We will now relate our problem to the relative grading between $\xx'$ and $\yy'$.

We can decompose $Y'$ as $Y'=Y\cup_F (F\times[0,1])\cup_F \hat{Y}$, where $F\times[0,1]$ is the intersection of a neighborhood of $\partial Y$ with $Y'\setminus Y$. We can assume without loss of generality that $B'\subset Y\cup (F\times[0,1])$ and that the unstable manifolds of all the index one critical points are $[0,1]$-invariant in $F\times[0,1]$. Following our construction of the gradings, let $v_{\xx}$, $v_{\yy}$ be the vector fields whose relative homotopy classes are $\gr(\xx)$ and $\gr(\yy)$, and let $v_{(s_1,\rho_1)},\dots,v_{(s_q,\rho_q)}$ be the vector fields defined in \S\ref{sec:alg} such that $[v_{(s_i,\rho_i)}]=\gr(I(s_i)a(\rho_i))$.
Let $\t=I_A(\yy)$ and let $\mathbb{I}_{\t}$ denote the $[0,1]$-invariant vector field on $F\times[0,1]$ whose restriction to $F\times\{t\}$ equals $v_{\t}$. Then the action of $[\mathbb{I}_{\t}]$ on $\Vect(Y,v_{\t})$ is trivial. So $[v_{\yy}\cdot\mathbb{I}_{\t}]=\gr(\yy)$. 
Therefore, in order to prove~\eqref{grading:eq}, it is enough to show that \begin{equation}\label{eq:vind}[v_{\xx}\cdot (v_{(s_1,\rho_1)}\cdots v_{(s_q,\rho_q)})]=\lambda^{\textrm{ind}(B,\vec{\brho})-l}\cdot [v_{\yy}\cdot\mathbb{I}_{\t}].\end{equation}

Since $v_{(\s,\vec{\rho})}$ and $\mathbb{I}_{\t}$ coincide on $F\times\{1\}$, we can extend $v_{\xx}\cdot(v_{(s_1,\rho_1)}\cdots v_{(s_q,\rho_q)})$ and $v_{\yy}\cdot \mathbb{I}_{\t}$ to $Y'$ so that they coincide in $\hat{Y}$. Let $v_{\xx,\vec{\rho}}$ and $v_{\yy,\t}$ be the vector fields obtained by this extension from $v_{\xx}\cdot(v_{(s_1,\rho_1)}\cdots v_{(s_q,\rho_q)})$ and $v_{\yy}\cdot \mathbb{I}_{\t}$, respectively.
We apply Proposition~\ref{htpydiff}, obtaining a link denoted by $L_{(\xx,\vec{\rho}),\yy}$ defined as
$$L_{(\xx,\vec{\rho}),\yy}:=\{y\in Y'|v_{\xx,\vec{\rho}}(y)=-v_{\yy,\t}(y)\}.$$
Since $v_{\xx,\vec{\rho}}$ and $v_{\yy,\t}$ coincide in $\hat{Y}$, the link $L_{(\xx,\vec{\rho}),\yy}$ is contained in $Y\cup(F\times[0,1))$ and it is independent of the extension of the vector fields to $\hat{Y}$.

Let $v_{\xx'}$ and $v_{\yy'}$ be the vector fields on $Y'$ as contructed in~\cite[\S2]{gh} whose homotopy classes are $\widetilde{\textrm{gr}}(\mathbf{x}')$ and $\widetilde{\textrm{gr}}(\mathbf{y}')$, respectively. We define $L_{\xx',\yy'}$ to be the link in $Y'$ given by
$$L_{\xx',\yy'}=\{y\in Y'|v_{\xx'}(y)=-v_{\yy'}(y)\}.$$
We note that $v_{\xx'}|_Y=v_{\xx}$ and $v_{\yy'}|_Y=v_{\yy}$. So the restrictions of $L_{(\xx,\vec{\rho}),\yy}$ and $L_{\xx',\yy'}$ to $Y$ coincide. We observe that this is the union of the flow lines corresponding to all points in $\xx$ and $\yy$, up to a small isotopy in neighborhoods of the critical points. 

We will now show that $L_{(\xx,\vec{\rho}),\yy}$ and $L_{\xx',\yy'}$ are both nullhomologous and isotopic to each other. We first look at $L_{(\xx,\vec{\rho}),\yy}\cap (F\times[0,1])$. For each $i$, we can assume that $v_{(s_i,\rho_i)}$ is defined in $F\times[\frac{i-1}{q},\frac{i}{q}]$. We will now compare $v_{(s_i,\rho_i)}$ with $\mathbb{I}_{\t}$. Using the description of $v_{(s_i,\rho_i)}$ illustrated in Figure~\ref{chordgr}, $v_{(s_i,\rho_i)}$ is a vector field which is $t$-invariant outside $N(\hat{\rho}_i)\times[\frac{i-1}{q},\frac{i}{q}]$, where $N(\hat{\rho}_i)$ is a neighborhood of the union of $\rho_i$ with the Morse trajectories connecting its ends to the corresponding critical points, as in \S\ref{sec:alg}. It follows from~\eqref{eq:cond} that $M(\rho_i^-)\not\in\t$. Now either $M(\rho_i^+)\not\in\t$ or $M(\rho_i^+)\in\t$. In the first case, $L_{(\xx,\vec{\rho}),\yy}\cap(F\times\{\frac{i-1}{q}\})=\{M(\rho_i^-)\}\times\{\frac{i-1}{q}\}$ and we obtain an arc that is always transverse to $F\times\{t\}$ and follows the point labeled with ``$+$'' as it travels from one critical point to the other. In the second case, $L_{(\xx,\vec{\rho}),\yy}\cap(F\times\{\frac{i-1}{q}\})=\{M(\rho_i^-), M(\rho_i^+)\}\times\{\frac{i-1}{q}\}$ and we obtain an arc which follows the points labeled with ``$+$'' and ``$-$'' until the middle of the second bifurcation, where the two parts of the arc connect. See Figure~\ref{fig:L}(a),(b) for an illustration of both cases. So $L_{(\xx,\vec{\rho}),\yy}\cap (F\times[0,1])$ is the union of these arcs for $i=1,\dots,q$. Note also that $L_{(\xx,\vec{\rho}),\yy}$ does not intersect $F\times\{1\}$.

Now we look at $L_{\xx',\yy'}\setminus Y$. We observe that $L_{\xx',\yy'}\setminus Y$ is, up to a slight perturbation, the union of the flow lines corresponding to the intersection points in $\xx'$ and $\yy'$ in $\Sigma'\setminus \Sigma$, except for the points on $\beta_i'$ for $i>q$. See Figure~\ref{fig:L}(c) for an example of $L_{\xx',\yy'}\setminus Y$. For each Reeb chord $\rho_i$, we can isotope the corresponding arc in $L_{(\xx,\vec{\rho}),\yy}\cap(F\times[\frac{i-1}{q},\frac{i}{q}])$ to have endpoints near $\partial Y$. We can also isotope $L_{\xx',\yy'}\setminus Y$ along the stable manifold of the index two critical point corresponding to $\beta_i'$ so that it is contained in $F\times[0,1]$. Because we chose $B'$ acording to the order of the Reeb chords in $\vec{\rho}$, it follows that we can perform a relative isotopy in $Y'\setminus Y$ so that $L_{(\xx,\vec{\rho}),\yy}\cap(F\times[0,1])$ is mapped to $L_{\xx',\yy'}\setminus Y$.

\begin{figure}
 \centering
\begin{minipage}[b]{0.4\textwidth}
\def\svgwidth{\textwidth}
 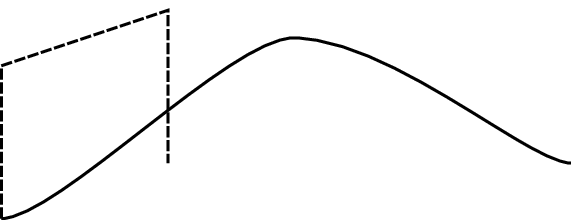
\subcaption{}
\end{minipage}\qquad
\begin{minipage}[b]{0.23\textwidth}
\def\svgwidth{\textwidth}
 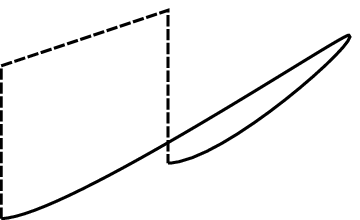
\subcaption{}                      
\end{minipage}\qquad
\begin{minipage}[b]{0.23\textwidth}
\def\svgwidth{\textwidth}
 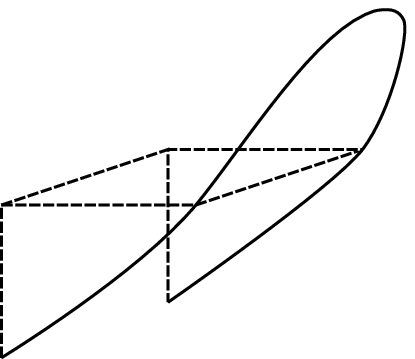
\subcaption{}
\end{minipage}
\caption{The two cases for an arc in $L_{(\xx,{\protect\vec{\rho}}),\yy}$ and an arc in $L_{\xx',\yy'}$.}
\label{fig:L}
\end{figure}

Let $C'$ be a 2-chain in $Y'$ obtained by taking the union of $B'$ with the stable and unstable submanifolds corresponding to all segments contained in $\partial B'$. Then the boundary of $C'$ is the union of the flow lines corresponding to $\xx'$ and $\yy'$, where the orientations on the flow lines corresponding to $\xx'$ and $\yy'$ are opposite to each other, so that the union of these flow lines is a closed curve. In particular, up to a small isotopy $\partial C'=L_{\xx',\yy'}$. So $L_{\xx',\yy'}$, and consequently $L_{(\xx,\vec{\rho}),\yy}$, are nullhomologous.

Let $n_{(\xx,\vec{\rho}),\yy}$ and $n_{\xx',\yy'}$ denote the framings on $L_{(\xx,\vec{\rho}),\yy}$ and $L_{\xx',\yy'}$ as in \S\ref{sec:htpy}. Our goal is to compute $n_{(\xx,\vec{\rho}),\yy}$. It follows from Proposition~\ref{htpydiff} and~\eqref{eq:grcf} that $n_{\xx',\yy'}=\textrm{ind}(B')$. So it is enough to compute $n_{(\xx,\vec{\rho}),\yy}-n_{\xx',\yy'}$.

We claim that 
\begin{equation}
 n_{(\xx,\vec{\rho}),\yy}-n_{\xx',\yy'}=-q.\label{eq:diff}
\end{equation}

\textsc{Step 4}: We now prove the proposition under our assumption.

We observe that $C'\cap \partial Y$ is the union of the curves $\hat{\rho}_i$ defined in \S\ref{sec:alg}. We can now choose Seifert surfaces $S_1$ and $S_2$ for $L_{(\xx,\vec{\rho}),\yy}$ and $L_{\xx',\yy'}$, respectively such that $S_1\cap Y=S_2\cap Y$, which both coincide with $C'\cap \partial Y$ on $\partial Y$. We can choose $S_2$ so that $S_2\setminus Y$ is a slight pertubation of $C'\setminus Y$ in $Y'\setminus Y$. We can assume, without loss of generality, that $S_1\cap (F\times[0,1])\subset (\cup_i N(\hat{\rho}_i))\times[0,1]$. Note also that an isotopy from $L_{(\xx,\vec{\rho}),\yy}$ to $L_{\xx',\yy'}$ induces an isotopy from $S_1$ to $S_2$. See Figure~\ref{fig:S} for an example of $S_1$ and $S_2$ in $Y'\setminus Y$.

\begin{figure}
 \centering
\begin{minipage}[b]{0.27\textwidth}
\def\svgwidth{\textwidth}
 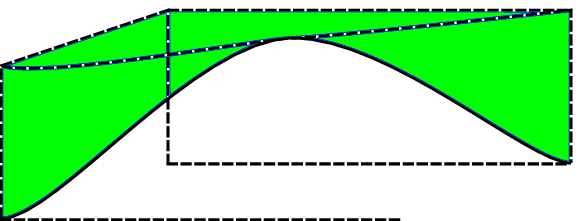
\end{minipage}\qquad
\begin{minipage}[b]{0.32\textwidth}
\def\svgwidth{\textwidth}
 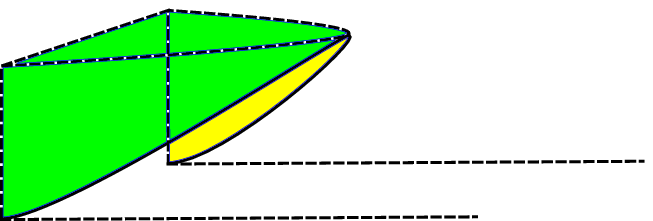
\end{minipage}\qquad
\begin{minipage}[b]{0.2\textwidth}
\def\svgwidth{\textwidth}
 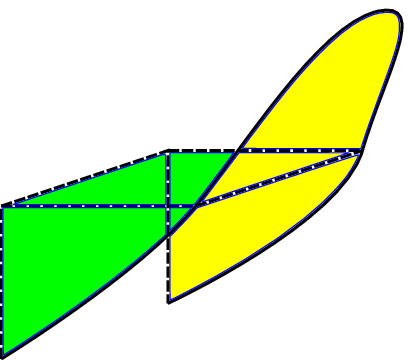
\end{minipage}
\caption{The surface $S_1$ in both cases and the surface $S_2$.}
\label{fig:S}
\end{figure}

In order to compute $n_{(\xx,\vec{\rho}),\yy}$ and $n_{\xx',\yy'}$, we first need to fix trivializations of $v_{\yy,\t}^\perp$ and $v_{\yy'}^\perp$ on neighborhoods of $S_1$ and $S_2$, respectively. We denote these neighborhoods  $N(S_1)$ and $N(S_2)$. Recall that $v_{\yy,\t}|_Y=v_{\yy'}|_Y$. We now fix a trivialization of $v_{\t}^{\perp}$ on $\cup_i N(\hat{\rho}_i)$ as follows. First recall that the trivialization $TF|_{N(\hat{\rho}_i)}\cong \R^2$ from the embedding in Figure~\ref{gradingbottom} and denote by $x$ and $y$ the coordinates on $\R^2$. We also write $\vec{\rho}^{\pm}=\{\rho_1^{\pm},\dots,\rho_l^{\pm}\}$. For each critical point $p\in M(\vec{\rho}^-\cup\vec{\rho}^+)$, we define a vector field $X_p$ in a neighborhood of $p$ as follows. If $p=M(\rho_i^-)$ for some $i$, we let $X_p=-\partial/\partial x$, where we identify $TF|_{N(\hat{\rho}_i)}\cong \R^2$. If $p=M(\rho_i^+)$ for some $i$, we let $X_p=\partial/\partial x$, where we identify $TF|_{N(\hat{\rho}_i)}\cong \R^2$. Note that if $p=M(\rho_i^+)=M(\rho_j^-)$ for $i<j$, then $X_p$ is well-defined, since we are assuming that $\rho_j^-\neq\rho_i^+$.

Let $t$ denote the $[0,1]$-coordinate in $F\times[0,1]$. As usual, we can think of $\partial/\partial t$ as a vector field on $F$. Let $V$ denote the vector field on $\cup_i N(\hat{\rho}_i)$ defined by $$V=\sum_{p\in M(\vec{\rho}^-\cup\vec{\rho}^+)}\phi_p\cdot \Bigg(X_p-\frac{\partial}{\partial t}\Bigg) +\frac{\partial}{\partial t},$$ where $\phi_p$ is a bump function such that $\phi_p=1$ at $p$ and $\phi_p=0$ outside a small neighborhood of $p$. We note that the vector field $v_{\t}$ is never tangent to $V$ along $N(\hat{\rho}_i)$ for every $i$. Hence the orthogonal projection of $V$ to $v_{\t}^{\perp}$ is a nonvanishing section of $v_{\t}^{\perp}|_{\cup_iN(\hat{\rho}_i)}$, giving rise to a trivialization of $v_{\t}^{\perp}|_{\cup_iN(\hat{\rho}_i)}$ for each $i$, where this section is identified with $(1,0,0)$. We extend this trivialization arbitrarily to a trivialization $\tau_Y$ of $v_{\yy,\t}^{\perp}|_{N(S_1\cap Y)}$, where $N(S_1\cap Y)$ is a neighborhood of $S_1\cap Y$ in $Y$. Since $S_1\cap (F\times[0,1])\subset(\cup_iN(\hat{\rho}_i))\times[0,1]$ and since $v_{\yy,\t}$ is $t$-invariant, we can extend $\tau_Y$ to a trivialization $\tau_{S_1}$ of $v_{\yy,\t}^{\perp}|_{N(S_1)}$, which is $t$-invariant on $N(S_1)\cap (F\times[0,1])$.

We now extend $\tau_F$ to a trivialization of $v_{\yy'}^\perp|_{N(S_2)}$ as follows. We will first extend $V$ to a vector field on a neighborhood of $S_2\setminus Y$, denoted by $N(S_2\setminus Y)$. Let $\epsilon>0$ such that $L_{\xx',\yy'}\cap (F\times\{\epsilon\})\neq\emptyset$. For $x\in L_{\xx',\yy'}\cap (F\times[0,\epsilon])$, we set $V(x)=V(\pi(x))$, where $\pi$ is the projection onto $F$. We can choose $\epsilon$ so that for $x\in L_{\xx',\yy'}\cap (F\times\{\epsilon\})$, we have $V(x)=\partial/\partial t$. Now, for $x\in L_{\xx',\yy'}\setminus (Y\cup(F\times[0,\epsilon]))$,  we set $V(x)$ to be always perpendicular to the stable manifold of the corresponding index two critical point. It follows from our construction of $V$  that we can extend $V$ to a vector field in $N(S_2\setminus Y)$ that is never parallel to $v_{\yy'}$. The orthogonal projection of $V$ onto $v_{\yy'}^\perp|_{N(S_2\setminus F)}$ induces an extension of $\tau_F$ to a trivialization of $v_{\yy'}^\perp|_{N(S_2)}$, denoted by $\tau_{S_2}$.

We are now ready to compute the framings on $L_{(\xx,\vec{\rho}),\yy}$ and $L_{\xx',\yy'}$. Using the trivializations $\tau_{S_1}$ and $\tau_{S_2}$ respectively, we can see the vector fields $v_{\xx,\vec{\rho}}$ and $v_{\xx'}$ as maps $v_{\xx,\vec{\rho}}:N(S_1)\to S^2$ and $v_{\xx'}:N(S_2)\to S^2$. We consider the links $K_1=v_{\xx,\vec{\rho}}^{-1}(\delta,0,-\sqrt{1-\delta^2})$ and $K_2=v_{\xx'}^{-1}(\delta,0,-\sqrt{1-\delta^2})$. It follows from our contruction that these links coincide in $Y$. Figure~\ref{fig:K} shows a picture of the pieces of $K_1$ and $K_2$ corresponding to a Reeb chord $\rho_i$. We observe that $K_1\setminus Y$ intersects $S_1\setminus Y$ once and $K_2\setminus Y$ does not intersect $S_2\setminus Y$. The intersection of $K_1\setminus Y$ with $S_1\setminus Y$ is negative by our sign convention. Therefore, we have shown~\eqref{eq:diff}. Combining~\eqref{eq:B'} and~\eqref{eq:diff}, we conclude that~\eqref{eq:vind} holds.

\begin{figure}
 \centering
\begin{minipage}[b]{0.3\textwidth}
\def\svgwidth{\textwidth}
 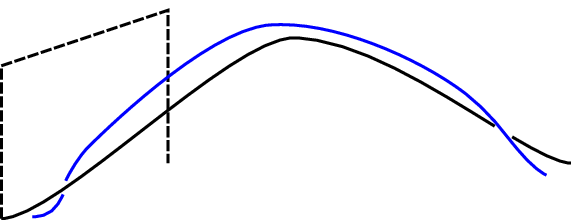
\end{minipage}\qquad
\begin{minipage}[b]{0.2\textwidth}
\def\svgwidth{\textwidth}
 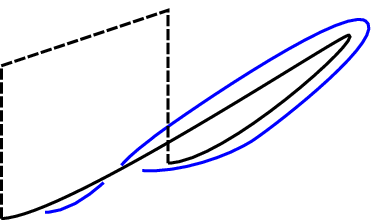
\end{minipage}\qquad
\begin{minipage}[b]{0.25\textwidth}
\def\svgwidth{\textwidth}
 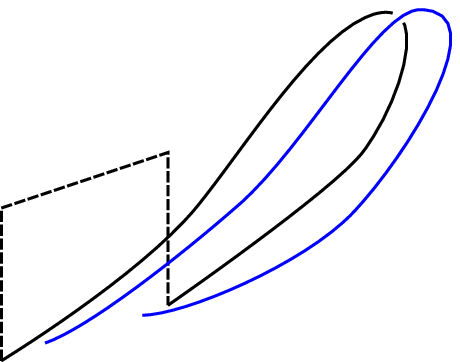
\end{minipage}
\caption{The framing $K_1$ in both cases and the framing $K_2$.}
\label{fig:K}
\end{figure}

\textsc{Step 5}: We now prove the proposition in the general case.

We again let $\vec{\rho}=(\rho_{1,1},\dots,\rho_{1,|\brho_1|},\dots,\rho_{l,1},\dots,\rho_{l,|\brho_l|})$ and $q=|\brho_1|+\dots+|\brho_l|$. Recall that $\iota(\vec{\brho})=\iota(\vec{\rho})$. We define a new Heegaard diagram $(\Sigma',\bm\alpha',\bm\beta',z)$ as follows. The surface $\Sigma'$ is a genus $(g+\max(k,q))$ surface containing $\Sigma$. We add $\max(k,q)$ $\beta$-circles to $\bm\beta$ to obtain $\bm\beta'$ and we denote the new $\beta$-circles by $\beta_i'$. We choose the $\beta$-circles in $\Sigma'\setminus\Sigma$ to be completions of parallel copies of $Z\setminus N(z)$ as in Step 2. To obtain $\bm\alpha'$ from $\bm\alpha$, we close the $\alpha$-arcs and, if $q>k$, we add $(q-k)$ $\alpha$-circles. We denote by $Y'$ the closed three-manifold obtained from $(\Sigma',\bm\alpha',\bm\beta',z)$. We again write $\vec{\rho}=(\rho_1,\dots,\rho_q)$. For each $\rho_i$, we obtain a segment $b_i$ on $\beta_i'$ by translating $\rho_i$ to $\beta_i'$. Now for each $\rho_i$, we extend $B$ into $\Sigma'\setminus\Sigma$ until its boundary hits $b_i$, see Figure~\ref{domain}. As in Step 2, we obtain a domain $B'$. Let $\xx'$ and $\yy'$ be the union of the corresponding intersection points. If~\eqref{eq:cond} does not hold, then $\xx'$ (and consequently $\yy'$) is not a generator of $\widehat{\mathit{CF}}(\Sigma',\bm\alpha',\bm\beta',z)$, since it contains intersection points on the same $\alpha$-circle. Nevertheless $\text{ind}(B')$ can still be defined using the combinatorial formula~(\ref{combind}).

\begin{figure}[h]
    \begin{overpic}[scale=.33]{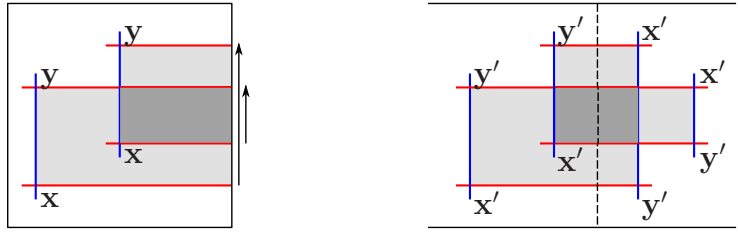}
    \put(4.5,3){$\mathbf{x}$}
    \put(16,9){$\mathbf{x}$}
    \put(16,26){$\mathbf{y}$}
    \put(4.5,20){$\mathbf{y}$}
    \put(64,2){$\mathbf{x}'$}
    \put(87,2){$\mathbf{y}'$}
    \put(87,26){$\mathbf{x}'$}
    \put(75.5,26){$\mathbf{y}'$}
    \put(75.5,8){$\mathbf{x}'$}
    \put(95,8.5){$\mathbf{y}'$}
    \put(95,20){$\mathbf{x}'$}
    \put(64,20){$\mathbf{y}'$}
    \end{overpic}
    \caption{The left side is a domain on $\Sigma$. The right side is the completion of the domain on $\Sigma'$.}
    \label{domain}
\end{figure}

Let $v_{\xx}$, $v_{\yy}$ and $v_{(\s_i,\brho_i)}$ be the vector fields defined in Section~\ref{construction:G} whose homotopy classes are $\gr(\xx)$, $\gr(\yy)$ and $\gr(I(\s_i)a(\brho_i))$, respectively. Let $v_{(\s,\vec{\brho})}=v_{(\s_1,\brho_1)}\cdots v_{(\s_l,\brho_l)}$. We again decompose $Y'$ as $Y'=Y\cup (F\times[0,1])\cup \hat{Y}$ and we consider the $[0,1]$-invariant vector field $\mathbb{I}_{\t}$ in $F\times[0,1]$. As in Step 2, we can extend $v_{\xx}\cdot v_{(\s,\vec{\brho})}$ and $v_{\yy}\cdot\mathbb{I}_{\t}$ to nonvanishing vector fields on $Y'$ which coincide in $\hat{Y}$, denoted by $v_{\xx,\vec{\brho}}$ and $v_{\yy,\t}$, respectively. So we need to prove that
\begin{equation}
[v_{\xx,\vec{\brho}}]=\lambda^{\text{ind}(B,\vec{\brho})-l}\cdot [v_{\yy,\t}].\label{eq:lambda}
\end{equation}
We again define $L_{(\xx,\vec{\brho}),\yy}=\{y\in Y'|v_{\xx,\vec{\brho}}(y)=v_{\yy,\t}(-y)\}.$ For each Reeb chord $\rho_i$, there is a corresponding arc in $L_{(\xx,\vec{\brho}),\yy}\cap (F\times[0,1])$. There are a few more cases to consider than the two cases in Figure~\ref{fig:L}, but in all cases the projection of the arc to $F$ is a slight perturbation of the union of $\hat{\rho}_i$. Moreover, we can assume that the arc corresponding to $\rho_i$ is contained in $F\times[\frac{i-1}{q},\frac{i}{q}]$. So $L_{(\xx,\vec{\brho}),\yy}\cap (F\times[0,1])$ is the union of all these arcs. We can perturb $L_{(\xx,\vec{\brho}),\yy}\cap (F\times[0,1])$ so that all the intersections of its projection to $F$ are transverse. 

Fix a trivialization $\tau$ of $v_{\yy,\t}^{\perp}$ in a small neighborhood of a Seifert surface of $L_{(\xx,\vec{\brho}),\yy}\cap (F\times[0,1])$, which is $[0,1]$-invariant in $F\times[0,1]$. Let $n_{(\xx,\vec{\brho}),\yy}$ denote the framing on $L_{(\xx,\vec{\brho}),\yy}\cap (F\times[0,1])$ obtained from $v_{\xx,\vec{\brho}}$ and $\tau$ as in Proposition~\ref{htpydiff}. By~\eqref{eq:lambda}, it is enough to prove that 
\begin{equation}
n_{(\xx,\vec{\brho}),\yy}=e(B)+n_{\xx}(B)+n_{\yy}(B)+\iota(\vec{\brho}).\label{eq:nxx}
\end{equation}

For a segment $b_i$, we denote the projection of its endpoints to $Z\setminus z$ by $b_i^-$ and $b_i^+$, where $b_i^-<b_i^+$. We note that the segment $b_i$ is specified by the projection of its endpoints and the $\beta$-circle to which it belongs. We say that a pair of segments $\{b_i,b_j\}$ is positively (resp. negatively) interleaved if $b_i^-<b_j^-<b_i^+<b_j^+$ and $i<j$ (resp. $j<i$). We say that $\{b_i,b_j\}$ is positively (resp. negatively) nested if $b_i^-<b_j^-<b_j^+<b_i^+$ and $i<j$ (resp. $j<i$). Finally, we say that $\{b_i,b_j\}$ is positively (resp. negatively) abutting if $b_i^+=b_j^-$ and $i<j$ (resp. $j<i$).

We define a new domain $B''$ by modifying $B'$. At each step, we denote by $b_1,\dots,b_n$ the segments of the corresponding domain on $\beta_i'$, where $n$ may vary after each modification. We will first remove all abutting and interleaved pairs of segments. Let $\{b_i,b_{i+1}\}$ be a pair of segments with consecutive indices. If $\{b_i,b_{i+1}\}$ is positively (resp. negatively) abutting, we substitute this pair by $[b_i^-,b_{i+1}^+]$ (resp. $[b_{i+1}^-,b_i^+]$) on $\beta_i'$. We now move the segments $b_j$ from $\beta_j'$ to $\beta_{j-1}'$ for $j>i+1$. In particular, substituting a pair of consecutive abutting pairs decreases the value of $n$ by $1$. If $\{b_i,b_{i+1}\}$ is interleaved, we substitute it by a nested pair. We can perform these changes until there are no abutting or interleaved pairs with consecutive indices. Now, let $\{b_i,b_j\}$ be an abutting or interleaved pair with $i<j$, such that $\{b_{i},\dots,b_j\}$ does not contain another abutting or interleaved pair. We note that either every pair of segments in $\{b_{i+1},\dots,b_j\}$ is nested or disjoint. We switch the segments $b_{i+1}$ and $b_j$. We proceed as above with the pair $\{b_i,b_{i+1}\}$. We repeat this procedure until there are no more abutting or interleaved pairs. Finally, we switch the order of the segments so that there are no negatively nested pairs. Let $b_1,\dots,b_m$ be the segments obtained from this procedure and let $B''$ be the domain resulting from these segments.

After each step of the above procedure, we obtain a new tuple of Reeb chords, for which we can compute the value of $\iota$. We observe that each time that a negatively abutting pair is concatenated in the above procedure, the value of $\iota$ increases by 1. Concatenating a positively abutting pair does not change $\iota$. For each positively (resp. negatively) interleaved pair that is exchanged by a nested pair, the value of $\iota$ is changed by $-1$ (resp. $+1$). Let $A^-$ denote the number of negatively abutting pairs that were concatenated in the above procedure and let $I^+$ (resp. $I^-$) denote the number of positively (resp. negatively) interleaved pairs that were exchanged by nested pairs in the above procedure. Hence
\begin{align}
 \text{ind}(B'')&=e(B)+n_{\xx}(B)+n_{\yy}(B)+\frac{m}{2},\label{eq:6a}\\
\iota(\vec{\rho})&=-\frac{m}{2}-A^-+I^+-I^-.\label{eq:6b}
\end{align}
Let $\xx''$ and $\yy''$ be the sets of intersection points corresponding to the corners of $B''$. It is still possible that $\xx''$ contains intersection points on the same $\alpha$-circle. That will happen if and only if $\yy''$ contains intersection points on the same $\alpha$-circle. Let us assume first that this does not occur. In this case, for each $i>m$, we choose an intersection point on $\beta_i'$ to add to both $\xx''$ and $\yy''$ so that $\xx''$ and $\yy''$ are generators of $\widehat{\mathit{CF}}(\Sigma',\bm\alpha',\bm\beta',z)$.

We define $L_{\xx'',\yy''}$ and $n_{\xx'',\yy''}$ as in Step 3. We note that $L_{(\xx,\vec{\brho}),\yy}$ and $L_{\xx'',\yy''}$ are cobordant. In fact, for each step of the above procedure, we perform a corresponding $0$-surgery or an isotopy as follows. We start from $L_{(\xx,\vec{\brho}),\yy}$. For a negatively abutting pair that is concatenated in the above procedure, we perform a positive $0$-surgery to the link. For each positively (resp. negatively) interleaved pair that is exchanged by a nested pair, we perform a negative (resp. positive) 0-surgery to the link. Finally, when we concatenate a positively abutting pair or when we exchange a nested or disjoint pair, we simply isotope the link. The resulting link can now be isotoped to $L_{\xx'',\yy''}$. Let $S_2$ be a Seifert surface for $L_{\xx'',\yy''}$, as in Step 4. We can choose an embedded Seifert surface $S_2$ for $L_{\xx'',\yy''}$ which is a slight pertubation of the union of $B''$ with the corresponding stable and unstable submanifolds, such that $S_2\cap Y=S_1\cap Y$. So we can extend the trivialization of $v_{\yy,\t}^{\perp}|_{N(S_1\cap Y)}$ to a trivialization $\tau_{S_2}$ of $v_{\yy''}^\perp|_{N(S_2)}$ and we obtain a link $K_2$, as in Step 4. We observe that $K_2\setminus Y$ and $S_2\setminus Y$ do not intersect, as before. Using the cobordism from $L_{(\xx,\vec{\rho}),\yy}$ to $L_{\xx'',\yy''}$, the link $K_1$ induces a link $\tilde{K}_1$ about $L_{\xx'',\yy''}$. For each $1\le i\le m$, we obtain a negative intersection of $\tilde{K}_1$ and $S_2$. Moreover, for each $0$-surgery that we performed, we obtain an extra intersection. Therefore
\begin{equation}
 n_{(\xx,\vec{\brho}),\yy}-n_{\xx'',\yy''}=-m-A^++I^+-I^-.\label{eq:diff'}
\end{equation}
Since $\text{ind}(B'')=n_{\xx'',\yy''}$ and $\iota(\vec{\brho})=\iota(\vec{\rho})$, it follows from~\eqref{eq:6a},~\eqref{eq:6b} and~\eqref{eq:diff'} that~\eqref{eq:nxx} holds.

Now assume that $\xx''$ contains intersection points on the same $\alpha$-circle. We modify the Heegaard diagram and the domain $B''$ as follows. We first note that any $\alpha$-circle contains as many points in $\xx''$ as in $\yy''$. Let $r_i$ denote the number of points in $\xx''$ contained in $\alpha_i'$. For each $i$ such that $r_i>1$, let $\Sigma_i$ be a surface of genus $(r_i-1)$ and let $d_i$ be a point in $\alpha_i'\cap \hat{Y}$. We consider the connect sum of $\Sigma'$ with the surfaces $\Sigma_i$, where the connect sum with $\Sigma_i$ is performed by removing a small disk centered at $d_i$. Now we add $(r_i-1)$ $\alpha$-circles, which are translates of $\alpha_i$ in $\Sigma'$ and are given by the model of Figure~\ref{fig:surface} on $\Sigma_i$. We also add $(r_i-1)$ $\beta$-circles on $\Sigma_i$, as in Figure~\ref{fig:surface}. We isotope these $\beta$-circles so that they intersect all $\alpha$ circles in $\Sigma'$. After this modification, we obtain a new Heegaard diagram for $Y'$. We can now move the points which lie in both $\xx''$ and $\yy''$, which are on the same $\alpha$-circles to distinct ones and we obtain\footnote{We also have to add intersection points to both $\xx''$ and $\yy''$ on all $\alpha$ and $\beta$-circles, which do not contain an intersection point yet. We choose such points in the complement of $B'''$.} $\xx'''$, $\yy'''$ and a new domain $B'''\in\pi_2(\xx''',\yy''')$ in the new Heegaard diagram. We observe that we can choose $\xx'''$, $\yy'''$ and $B'''$ so that $\text{ind}(B'')=\text{ind}(B''')$. 
If $o(\xx)\cap o(\yy))\setminus o(\xx \cap\yy)=\emptyset$, then the argument from the above paragraph works if we exchange $\xx''$, $\yy''$ and $B''$ by $\xx'''$, $\yy'''$ and $B'''$. If $o(\xx)\cap o(\yy))\setminus o(\xx \cap\yy)\neq\emptyset$, then $L_{(\xx,\vec{\brho}),\yy}$ and $L_{\xx''',\yy'''}$ do not coincide in $Y$, since $\yy$ is not a subset of $\yy'''$. In this case, we write $\widetilde{\yy}=\yy'''\cap \Sigma$ and we denote by $\vec{\bm\sigma}$ the modification of $\vec{\brho}$ obtained by substituting $\yy$ by $\tilde{\yy}$. We then observe that the links $L_{(\xx,\vec{\brho}),\yy}$ and $L_{(\xx,\vec{\bm\sigma}),\tilde{\yy}}$ are framed homotopic, so~\eqref{eq:nxx} holds in all cases.

\vspace{1cm}

\begin{figure}[htb]
\centering \def\svgwidth{100pt}
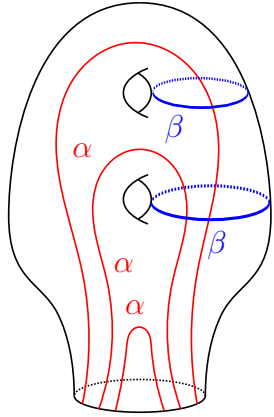 
\caption{The $\alpha$ and $\beta$-curves on the surface $\Sigma_i$.}\label{fig:surface}
\end{figure}

\end{proof}

Theorem~\ref{thm:mod}(a) is an immediate corollary of Proposition~\ref{index}.

\subsection{The grading on $\widehat{\mathit{CFD}}(\H)$}
We start by recalling the definition of the module $\widehat{\mathit{CFD}}(\H)$.
For $\xx\in\mathfrak{S}(\H)$, let $\bar{o}(\xx)=[2k]\setminus o(\xx)$ and define $I_D(\xx)=I([2k]\setminus o(\xx))$. We have a left action of the set of idempotents $\mathcal{I}$ on $\mathfrak{S}(\mathcal{H})$ given by \begin{equation*}
 I(\s)\cdot \xx=\left\{\begin{array}{ll}
                 \xx,&\text{if }I_D(\xx)= I(\s),\\
		 0,&\text{otherwise}.
                \end{array}\right.
\end{equation*}
The module $\widehat{\mathit{CFD}}(\H)$ is generated over $\Z/2$ by the elements of the form $a\otimes \xx$, where $a\in \A(-\ZZ)$ and $\xx\in\mathfrak{S}(\H)$, and the tensor is taken over $\mathcal{I}$. Its module structure is given by the obvious left $\A(-\ZZ)$-action.

We can define the grading $\gr$ on a generator $a(-\brho)\otimes\xx$ of $\widehat{\mathit{CFD}}(\H)$ by
$$\gr(a(-\brho)\otimes\xx):=\gr(a(-\brho)I_D(\xx))\cdot \gr(\xx).$$
The differential $\partial$ on $\widehat{\mathit{CFD}}(\H)$ is defined in~\cite{lot} by counting moduli spaces of holomorphic curves of the form $\mathcal{M}^B(\xx,\yy,\vec{\rho})$, where $\vec{\rho}=(\rho_1,\dots,\rho_l)$ is a sequence of Reeb chords. More precisely $\partial(I_{D}(\xx)\otimes \xx)$ is a sum of terms of the form $a(-\vec{\rho})\otimes \yy$, where $B\in\pi_2(\xx,\yy)$ and $\text{ind}(B,\vec{\rho})=1$. Here $-\vec{\rho}$ denotes $(-\rho_1,\dots,-\rho_l)$ and $a(-\vec{\rho})$ denotes the product $a(-\rho_1)\dots a(-\rho_l)$.

\begin{prop}
Let $\mathbf{x},\mathbf{y} \in \mathfrak{S}(\mathcal{H})$, $B\in\pi_2(\xx,\yy)$ and $\vec{\rho}$ such that $\partial^{\partial} B=[\vec{\rho}]$. If $a(-\vec{\rho})\otimes\yy\neq 0$, then \begin{equation}\label{eq:grcfd}
\text{\em gr}(a(-\vec{\rho})I_D(\yy))\cdot\text{\em gr}(\mathbf{y})= \lambda^{- \text{\em ind}(B,\vec{\rho})} \text{\em gr}(\xx).\end{equation}
\end{prop}

\begin{proof}
The proof is very to similar to that of Proposition~\ref{index}. We assume, for simplicity, that~\eqref{eq:cond} holds and that the Reeb chords are all pairwise disjoint. Otherwise, we can apply similar arguments to Step 5 of the proof of Proposition~\ref{index}.
 
We again construct a closed manifold $$Y'=\hat{Y}\cup_{\bar{F}}\cup\bar{F}\times[0,1]\cup_{\bar{F}}\bar{Y}.$$
And we extend $\H$ to a Heegaard decomposition of $Y'$ so that the new $\beta$-curves are translates of the Reeb chords. We again obtain generators $\xx'$, $\yy'$ of $\widehat{\mathit{CF}}(Y')$ and a homology class $B'\in\pi(\xx',\yy')$. So it follows from~\eqref{eq:B'} that $\text{ind}(B')=\text{ind}(B,\vec{\rho}).$
Now we compare $\gr(a(-\vec{\rho})I_D(\yy))\cdot \gr(\yy)$ and $[\mathbb{I}]\cdot \gr(\xx)$ in $(\bar{F}\times[0,1])$ where $\mathbb{I}$ is the $[0,1]$-invariant vector field which coincides with $\gr(\xx)$ along $\bar{F}\times\{1\}$. As in Step 3 of the proof of Proposition~\ref{index}, we obtain a link whose intersection with $Y'\setminus\bar{Y}$ is the union of arcs, one for each Reeb chord. We also consider a surface $S_1$ bounding this link and a link $K_1$ obtained by taking the preimage of $(\delta,0,-\sqrt{1-\delta^2})$, as in Step 4 of the proof above. We observe that, in this case, $K_1\setminus \bar{Y}$ does not intersect $S_1\setminus\bar{Y}$. So the framing equals $n_{\xx',\yy'}$. Therefore
$$\gr(\xx)=\lambda^{\text{ind}(B')}\cdot\gr(a(-\vec{\rho})I_D(\yy))\cdot \gr(\yy).$$
That implies our claim.
\end{proof}
We have therefore proven Theorem~\ref{thm:mod}(b).

\section{The pairing theorems}\label{sec:pair}
Our absolute grading is also compatible with the pairing theorems proved in~\cite{lot}. More precisely, given two bordered Heegaard diagrams $\mathcal{H}_1$ and $\mathcal{H}_2$ for $Y_1$ and $Y_2$, respectively, with $\partial\mathcal{H}_1 = -\partial\mathcal{H}_2$, we obtain a Heegaard diagram $\mathcal{H}=\mathcal{H}_1 \cup_\partial \mathcal{H}_2$ for the closed manifold $Y:=Y_1 \cup_\partial Y_2$. Let $F=\partial Y_1= -\partial Y_2$ be the parameterized boundary.

Recall that the box tensor product $\widehat{\mathit{CFA}}(Y_1)\boxtimes\widehat{\mathit{CFD}}(Y_2)$ is $\mathfrak{S}(\H_1)\otimes_{\mathcal{I}(\ZZ)}\mathfrak{S}(\H_2)$ as a set. See~\cite[Def. 2.26]{lot} for the definition of the differential. If $\xx_1\in\mathfrak{S}(\H_1)$ and $\xx_2\in\mathfrak{S}(\H_2)$, such that $\xx_1\otimes\xx_2\in \widehat{\mathit{CFA}}(Y_1)\boxtimes\widehat{\mathit{CFD}}(Y_2)$ is nonzero, then $\xx_1$ and $\xx_2$ must lie on complementary $\alpha$-arcs. Therefore the pair $(\xx_1,\xx_2)$ corresponds to a generator of $\widehat{\mathit{CF}}(Y)$. So there is a canonical map \begin{equation}\label{homequiv}\Phi:\widehat{\mathit{CFA}}(Y_1)\boxtimes\widehat{\mathit{CFD}}(Y_2)\to \widehat{\mathit{CF}}(Y).\end{equation} We recall the following theorem from~\cite{lot}.

\begin{thm}[{\cite[Thm. 1.3]{lot}}]
The map \eqref{homequiv} is a homotopy equivalence.
\end{thm}

Let $S(\H_1)\times_F S(\H_2)$ denote the set of elements of the form $([v_1],[v_2])$ with $[v_1]\in S(\H_1)$ and $[v_2]\in S(\H_2)$, such that $[v_1]$ and $[v_2]$ agree along $F$. Recall that $G(\ZZ_1)=G(-\ZZ_2)$ acts on $S(\H_1)$ on the right and on $S(\H_2)$ on the left. We now define $S(\H_1)\otimes_{G(\ZZ_1)} S(\H_2)$ to be the quotient of $S(\H_1)\times_F S(\H_2)$ by the equivalence relation given by $(\xi_1\cdot a, \xi_2)\sim (\xi_1,a\cdot\xi_2)$, where $\xi_i\in S(\H_i)$ for $i=1,2$ and $a\in G(\ZZ_1)$. Recall from~\cite{gh} that the absolute grading on $\widehat{\mathit{CF}}(Y)$ takes values in $\Vect(Y)$. Now given nonvanishing vector fields $v_1$ in $Y_1$ and $v_2$ in $Y_2$, which agree along $\partial Y_1 =-\partial Y_2$, we obtain a vector field $v_1 \cdot v_2$ on $Y$ by gluing along the boundary. Therefore we obtain a map $$\Psi:S(\H_1)\otimes_{G(\ZZ_1)} S(\H_2)\to \Vect(Y).$$
We have the following proposition.
\begin{prop}
 The map $\Psi$ is a bijection.
\end{prop}
\begin{proof}
To show that $\Psi$ is surjective, let $v$ be a nonvanishing vector field on $Y$ and write $v=v_1\cdot v_2$, where $v_1$ and $v_2$ are nonvanishing vector fields on $Y_1$ and $Y_2$, respectively. Now we fix a trivialization of $TY$, and hence a trivialization of $TY|_F$. By the Pontryagin-Thom construction, two maps $F\to S^2$ are isomorphic if, and only if, their pullbacks of the generator of $H^2(S^2;\Z)$ coincide. We observe that the pullback map $\iota^*:H^2(Y_1,\Z)\to H^2(F)$ is trivial. Hence $v_1|_F$ is homotopic to the constant map $F\to S^2$. Now fix $\s\subset[2k]$, such that $|\s|=k$. Since we can extend $v_{\s}$ to a vector field in $Y$, it follows that $v_{\s}$ is again homotopic to the constant map. Therefore there exists a nonvanishing vector field $u$ in $F\times[0,1]$ such that $u|_{F\times\{0\}}=v_1$ and $u|_{F\times\{1\}}=v_{\s}$. Let $\bar{u}$ denote the inverse of the homotopy determined by $u$. It follows that $v_1\cdot u\cdot \bar{u}\cdot v_2$ is homotopic to $v_1\cdot v_2$. So $\Psi([v_1\cdot u]\otimes [\bar{u}\cdot v_2])=[v_1\cdot v_2]$. Hence $\Psi$ is surjective.

Now let $[v_1],[w_1]\in S(\H_1)$ and $[v_2],[w_2]\in S(\H_2)$ such that $\Psi([v_1]\otimes[v_2])=\Psi([w_1]\otimes[w_2])$. So $[v_1\cdot v_2]=[w_1\cdot w_2]$ as elements in $\Vect(Y)$. Let $H:Y\times[0,1]$ denote the homotopy from $v_1\cdot v_2$ to $w_1\cdot w_2$. Let $u$ be the restriction of $H$ to $F\times[0,1]$. So $u|_{F\times\{0\}}=v_1|_F$ and $u|_{F\times\{1\}}=w_1|_F$. We observe that $[v_1\cdot u]=[w_1]\in S(\H_1)$ and that $[\bar{u}\cdot v_1]=[w_2]\in S(\H_2)$. So
$$[v_1]\otimes[v_2]=[v_1\cdot u]\otimes[\bar{u}\cdot v_2]=[w_1]\otimes[w_2]\in S(\H_1)\otimes_{G(\ZZ_1)} S(\H_2).$$
Therefore $\Psi$ is injective.
\end{proof}

We can now prove that the map \eqref{homequiv} preserves the absolute grading.

\begin{thm}
Given $\xx_1\in\mathfrak{S}(\H_1)$ and $\xx_2\in\mathfrak{S}(\H_2)$, such that $\xx_1\otimes\xx_2\neq 0$. Then
$$\widetilde{\text{\em gr}}(\Phi(\xx_1\otimes\xx_2))=\Psi(\text{\em \gr}(\xx_1)\otimes\text{\em \gr}(\xx_2)).$$
\end{thm}

\begin{proof}
This follows immediately from our construction of the gradings in \S3.3 and from the definition of the grading on Heegaard Floer homology in~\cite[\S2]{gh}.
\end{proof}

\vspace{10mm}
\noindent
{\em Acknowledgements}. We thank Robert Lipshitz for discussions on the groupoid and Patrick Massot for suggesting the work of Dufraine on homotopy classes of vector fields to us. We also thank the referees for very detailed comments and suggestions.

\end{document}